\documentclass[11pt]{amsart}

\usepackage{amscd,latexsym,amsthm,amsfonts,amssymb,amsmath,amsxtra}
\usepackage[mathscr]{eucal}
\usepackage{pictexwd,dcpic}
\usepackage[normalem]{ulem}
\usepackage{hyperref}
\usepackage{stmaryrd}

\renewcommand\theequation{\thesection.\arabic{equation}}

\newcommand{\BC}{{\mathbb {C}}}

\newcommand{\BN}{{\mathbb {N}}}

\newcommand{\BR}{{\mathbb {R}}}

\newcommand{\BZ}{{\mathbb {Z}}}

\newcommand{\CC}{{\mathcal {C}}}

\newcommand{\CL}{{\mathcal {L}}}

\newcommand{\CN}{{\mathcal {N}}}
\newcommand{\CO}{{\mathcal {O}}}

\newcommand{\CR}{{\mathcal {R}}}
\newcommand{\CS}{{\mathcal {S}}}
\newcommand{\CT}{{\mathcal {T}}}

\newcommand{\CX}{{\mathcal {X}}}

\newcommand{\Fa}{{\mathfrak {a}}}
\newcommand{\Fb}{{\mathfrak {b}}}

\newcommand{\Fg}{{\mathfrak {g}}}
\newcommand{\Fh}{{\mathfrak {h}}}

\newcommand{\Fl}{{\mathfrak {l}}}
\newcommand{\Fm}{{\mathfrak {m}}}
\newcommand{\Fn}{{\mathfrak {n}}}
\newcommand{\Fo}{{\mathfrak {o}}}

\newcommand{\Fs}{{\mathfrak {s}}}
\newcommand{\Ft}{{\mathfrak {t}}}

\newcommand{\Fz}{{\mathfrak {z}}}

\newcommand{\GL}{{\mathrm{GL}}}

\newcommand{\Hom}{{\mathrm{Hom}}}

\newcommand{\SL}{{\mathrm{SL}}}

\newcommand{\SO}{{\mathrm{SO}}}

\newcommand{\Sp}{{\mathrm{Sp}}}

\newcommand{\tr}{{\mathrm{tr}}}

\newcommand{\ul}{\underline}

\newcommand{\back}{\backslash}

\newtheorem{thm}{Theorem}[section]
\newtheorem{cor}[thm]{Corollary}
\newtheorem{lem}[thm]{Lemma}

\newtheorem{prop}[thm]{Proposition}
\newtheorem {conj}[thm]{Conjecture}

\newtheorem {ques/conj}[thm]{Question/Conjecture}

\newtheorem{defn}[thm]{Definition}
\newtheorem{rmk}[thm]{Remark}

\makeatletter

\newcommand{\Rmnum}[1]{\expandafter\@slowromancap\romannumeral #1@}
\makeatother

\begin{document}
\renewcommand{\theequation}{\arabic{equation}}
\numberwithin{equation}{section}

\title[Multiplicity formula]{On Multiplicity Formula for Spherical Varieties}

\author{Chen Wan}
\address{Department of Mathematics\\
Massachusetts Institute of Technology\\
Cambridge, MA 02139, USA}
\email{chenwan@mit.edu}

\begin{abstract}
In this paper, we propose a conjectural multiplicity formula for general spherical varieties. For all the cases where a multiplicity formula has been proved, including Whittaker model, Gan-Gross-Prasad model, Ginzburg-Rallis model, Galois model and Shalika model, we show that the multiplicity formula in our conjecture matches the multiplicity formula that has been proved. We also give a proof of this multiplicity formula in two new cases.
\end{abstract}



\maketitle

\tableofcontents

\section{Introduction}
Let $F$ be a local field of characteristic 0, $G$ be a connected reductive group defined over $F$, $H$ be a connected closed subgroup of $G$, and $\chi$ be a unitary character of $H(F)$. Assume that $H$ is a spherical subgroup of $G$ (i.e. $H$ admitting an open orbit in the flag variety of
$G$). For every irreducible smooth representation $\pi$ of $G(F)$, we define the multiplicity
$$m(\pi,\chi):=\dim(\Hom_{H(F)}(\pi,\chi)).$$
One of the fundamental problems in the {\it Relative Langlands Program} is to study the multiplicity $m(\pi,\chi)$. In general, one expects the multiplicity to be finite and to detect some functorial structures of $\pi$. We refer the readers to \cite{SV} for a detailed discussion of these kinds of problems.

In his pioneering works \cite{W10} and \cite{W12}, Waldspurger developed a new method to study the multiplicities. His idea is to prove a local trace formula $I_{geom}(f)=I(f)=I_{spec}(f)$ for the model $(G,H)$, which would imply a multiplicity formula $m(\pi,\chi)=m_{geom}(\pi,\chi)$. Here $m_{geom}(\pi,\chi)$ is defined via the Harish-Chandra character $\theta_{\pi}$ of $\pi$ and is called the geometric multiplicity. In his paper \cite{W10} and \cite{W12}, Waldspurger applied this method to the orthogonal Gan-Gross-Prasad models over p-adic field. By proving the trace formula and the multiplicity formula, he was able to show that for the orthogonal Gan-Gross-Prasad model, the summation of the multiplicities is always equal to 1 for all tempered local Vogan L-packets. Later his idea was adapted by Beuzart-Plessis \cite{B12}, \cite{B15} for the unitary Gan-Gross-Prasad model, and by the author \cite{Wan15}, \cite{Wan16} for the Ginzburg-Rallis model. Subsequently, in \cite{B18}, Beuzart-Plessis applied this method to the Galois model; in a joint work with Beuzart-Plessis \cite{BW}, we applied this method to the Shalika model; and in a joint work with Zhang \cite{WZ}, we applied this method to the unitary Ginzburg-Rallis model.

For all the cases above, the most crucial step in the proof is to prove the local trace formula $I_{geom}(f)=I(f)=I_{spec}(f)$. However, the proofs of these trace formulas, especially the geometric side (i.e. $I(f)=I_{geom}(f)$), have each time been done in some ad hoc way pertaining to the particular features of the case at hand. It makes now little doubt that the local trace formula and multiplicity formula should exist in some generality. However, until this moment, it is not clear (even conjecturally) what would both formulas look like for general spherical varieties. The reason is that although we can easily give a uniform definition of the multiplicity $m(\pi,\chi)$, the distribution $I(f)$ and the spectral expansion $I_{spec}(f)$ for all the spherical varieties, the geometric multiplicity $m_{geom}(\pi,\chi)$ and the geometric expansion $I_{geom}(f)$ are more mysterious. There are no uniform definitions of these two objects for general spherical varieties.

\begin{rmk}
\emph{The definitions of $m_{geom}(\pi,\chi)$ and $I_{geom}(f)$ are very similar to each other. So one only needs to define $m_{geom}(\pi,\chi)$ for general spherical varieties, which will lead to the definition of $I_{geom}(f)$.}
\end{rmk}

In this paper, we propose a uniform definition of $m_{geom}(\pi,\chi)$ (and hence $I_{geom}(f)$) for general spherical varieties. To justify our definitions, we show that for all the cases where the multiplicity formulas have been proved, including the Whittaker model, the Gan-Gross-Prasad model, the Ginzburg-Rallis model, the Galois model, and the Shalika model, our definition of the geometric multiplicity matches the one in the known multiplicity formula. We will also give a proof of the multiplicity formula for two new cases. We hope our definitions will give people a better understanding of the multiplicity formula and local trace formula, and shed some light on a potential proof of both formulas for general spherical varieties.

\subsection{Main results}
Let $F,G,H,\chi,m(\pi,\chi)$ be as above. Our goal is to define the geometric multiplicity $m_{geom}(\pi,\chi)$. Before we explain our definition, let's first consider the baby case when $G$ is a finite group. In this case, let $\theta_{\pi}(g)=\tr(\pi(g))$ be the character of $\pi$. By the representation theory of finite group, we know that $m(\pi,\chi)=m_{geom}(\pi,\chi)$ where
\begin{equation}\label{finite group case}
m_{geom}(\pi,\chi):=\frac{1}{|H|}\sum_{h\in H} \theta_{\pi}(h) \chi^{-1}(h)=\sum_{x} \frac{1}{|Z_H(x)|} \theta_{\pi}(h)\chi^{-1}(h).
\end{equation}
Here the second summation is over a set of representatives of conjugacy classes of $H$ and $Z_H(x)$ is the centralizer of $x$ in $H$.

Guided by the finite group case and all the known cases, it is natural to expect that for general spherical pair $(G,H)$, $m_{geom}(\pi,\chi)$ should be an integral over certain semisimple conjugacy classes of $H(F)$ of the Harish-Chandra character $\theta_{\pi}$. However, compared with the finite group case, there are three difficulties in the definition of $m_{geom}(\pi,\chi)$ for spherical varieties over local field.

First, unlike the finite group case, the Harish-Chandra character $\theta_{\pi}$ is only defined on the set of regular semisimple elements of $G(F)$. On the other hand, many semisimple conjugacy classes of $H(F)$ are not regular in $G(F)$ which means that $\theta_{\pi}$ is not defined in those conjugacy classes. In order to solve this issue, we need to use the germ expansion for $\theta_{\pi}$. Roughly speaking, near every semisimple element (not necessarily regular) of $G(F)$, $\theta_{\pi}$ can be written as a linear combination of the Fourier transform of the nilpotent orbital integrals. The coefficients associated to regular nilpotent orbits in this linear combination are called the regular germs of $\theta_{\pi}$ (see Section \ref{section shalika germ} for details). In order to define $\theta_{\pi}$ at non-regular semisimple conjugacy classes, we need to use the regular germs of $\theta_{\pi}$. This creates the first difficulty: in general when $F\neq \BC$, we may have more than one $F$-rational regular nilpotent orbits. Hence for each spherical pair $(G,H)$, we need to define a subset of regular nilpotent orbits whose regular germs appear in the geometric multiplicity. This will be done in Section \ref{Section nilpotent orbit} by using the conjugacy classes in the tangent space of $G/H$.

Secondly, we need to define the support (i.e. a subset of semisimple conjugacy classes of $H(F)$) of the geometric multiplicity. In the finite group case, the support of geometric multiplicity contains all the conjugacy classes of $H$. But this will not be the case for spherical varieties over local field. As we will see in Section \ref{Section torus}, the geometric multiplicity is only supported on those ``elliptic conjugacy classes" whose centralizers in $G(F)$ and $H(F)$ form a {\it minimal spherical variety} (see Section \ref{section spherical subgroup}) and whose centralizer in $G(F)$ is quasi-split. The quasi-split condition provides the existence of the regular germs, while the minimal spherical variety condition ensures that the ``homogeneous degree" of the spherical variety is equal to the homogeneous degree of the regular germs of the Harish-Chandra character. We refer the readers to Section \ref{Section torus} for details.

Thirdly, in the finite group case, we normalize the character $\theta_{\pi}$ by the number $\frac{1}{|Z_H(x)|}$. For general spherical varieties, we would need an extra number $d(G,H,F)$ which characterizes how the $G(\bar{F})$-conjugacy class (i.e. stable conjugacy class) in the tangent space of $G/H$ decomposes into $H(F)$-conjugacy classes. We refer the readers to Section \ref{Section constant d(G,H,F)} for details.

After we have solved the three difficulties above, we are able to write down the definition of $m_{geom}(\pi,\chi)$ (and hence $I_{geom}(f)$) for all spherical varieties in Section \ref{Section multiplicity formula}. We will state the conjectural multiplicity formula in Conjecture \ref{main conjecture}. In Section \ref{Section known cases}, we will show that for all the known cases, our definition of the geometric multiplicity $m_{geom}(\pi,\chi)$ matches the one in the known multiplicity formula.

\begin{thm}\label{main}
Assume that $F$ is p-adic. When $(G,H)$ is the Whittaker model, Gan-Gross-Prasad model, Ginzburg-Rallis model, Galois model, or Shalika model, the geometric multiplicity defined in Definition \ref{definition geometric multiplicity} matches the one in the multiplicity formula that has been proved. In particular, Conjecture \ref{main conjecture} holds for all these models.
\end{thm}

Our proof of Theorem \ref{main} uses a Lie algebra version of the local trace formula for Gan-Gross-Prasad model and Ginzburg-Rallis model, as well as a relation between the Shalika germ and Kostant section proved by Kottwitz (see Lemma \ref{Shalika germ}). In general if one can extend Lemma \ref{Shalika germ} to the archimedean case, then we can also prove Theorem \ref{main} when $F=\BR$ (the case when $F=\BC$ is trivial).

\begin{rmk}
\emph{Unlike the finite group case, we don't expect the multiplicity formula $m(\pi,\chi)=m_{geom}(\pi,\chi)$ holds for all irreducible smooth representations of $G(F)$. An easy example will be the model $(\GL_2,\GL_1)$. For this case, the geometric multiplicity is just the regular germ of $\theta_{\pi}$ at the identity element and one can show that the multiplicity formula holds for all generic representations. However, it is easy to see that this formula fails for nongeneric representations (i.e. finite dimensional representations) of $\GL_2(F)$.}

\emph{In general, the multiplicity formula should always hold for all supercuspdial representations. When the spherical pair is tempered, it should hold for all discrete series and for almost all tempered representations. When the spherical pair is strongly tempered, it should hold for all tempered representations.}

\emph{Moreover, as observed by Prasad in \cite{P}, if we want to make the multiplicity formula holds for all irreducible smooth representations of $G(F)$, we need to replace the multiplicity $m(\pi,\chi)$ by the Euler-Poincar\'e pairing $EP(\pi,\chi)$. We refer the readers to Section \ref{Section multiplicity formula} for details.}
\end{rmk}

Moreover, all of our discussions so far make sense when $\chi$ is a finite dimensional representations of $H(F)$. In particular, we can also define the geometric multiplicity $m_{geom}(\pi,\chi)$ when $\chi$ is a finite dimensional representation of $H(F)$ (when $F$ is p-adic, this is not interesting since finite dimensional representations of $H(F)$ are essentially characters).

The case we are interested in is when $F=\BR$ and $H(\BR)=K$ is a maximal connected compact subgroup of $G(\BR)$. In this case, $m(\pi,\chi)=m_{geom}(\pi,\chi)$ gives a multiplicity formula of $K$-types for all the irreducible smooth representations of $G(\BR)$ (note that since $H(\BR)$ is compact, we have $m(\pi,\chi)=EP(\pi,\chi)$ for all $\pi$). We refer the readers to Section \ref{section K-types} for more details. In Section 8 and 9, I will prove this multiplicity formula of K-types for $\GL_n(\BR)$ and for all the complex reductive groups.

\begin{thm}\label{main K-type}
The multiplicity formula of $K$-types (i.e. Conjecture \ref{conj K-type}) holds when
\begin{enumerate}
\item $G(F)=\GL_n(\BR)$.
\item $G=Res_{\BC/\BR}H$ is a complex reductive group.
\end{enumerate}
In particular, Conjecture \ref{main conjecture} holds for these two cases.
\end{thm}

The key ingredient of our proof of Theorem \ref{main K-type} is to show that both the multiplicity and the geometric multiplicity behave nicely under parabolic induction. For the multiplicity, this follows from the Iwasawa decomposition and the reciprocity law. For the geometric multiplicity, this follows from Proposition \ref{germ parabolic induction} which gives the behavior of the Harish-Chandra character under parabolic induction. After we have proved these arguments, we can use induction to finish the proof of Theorem \ref{main K-type}. The upshot is that when $G=\GL_n$ ($n>2$) or when $G=Res_{\BC/\BR}H$ is a nonabelian complex reductive group, the Grothendieck group of finite length smooth representations of $G(\BR)$ is generated by induced representations.

The paper is organized as follows: In Section 2, we introduce basic notation and conventions used in this paper. In Section 3, we will define a subset of conjugacy classes of $H(F)$, which will be the support of the geometric multiplicity. In Section 4, we introduce a constant $d(G,H,F)$ associated to minimal spherical varieties. It characterizes how the $G(\bar{F})$-conjugacy class in the tangent space of $G/H$ decomposes into $H(F)$-conjugacy classes. In Section 5, we define a subset of regular nilpotent orbits associated to minimal spherical varieties. The regular germs of those nilpotent orbits will show up in the geometric multiplicity. Then in Section 6, combining the works in Section 3-5, we will define the geometric multiplicity $m_{geom}(\pi,\chi)$ and the geometric expansion of the trace formula $I_{geom}(f)$ for general spherical varieties. In Section 7, we will show that for all the known cases, our definition of the geometric multiplicity matches the one in the known multiplicity formula. Finally, in Section 8 and 9, we will prove the multiplicity formula of $K$-types for $\GL_n(\BR)$ and for all the complex reductive groups.

\subsection{Acknowledgement}
I would like to thank Rapha\"el Beuzart-Plessis for the helpful comments on the first draft of this paper, and for many helpful discussions which lead to the definition of the geometric multiplicity when the spherical variety has Type N root. I would also like to thank an anonymous referee for the helpful comments and corrections on a previous version of this paper.

\section{Preliminary}
\subsection{Notation}
Let $F$ be a local field of characteristic 0, and $\psi:F\rightarrow \BC^{\times}$ be a nontrivial additive character. Let $G$ be a connected reductive group defined over $F$, $\Fg$ be the Lie algebra of $G$, $Z_G$ be the center of $G$, and $A_G(F)$ be the maximal split torus of $Z_G(F)$.  We use $G_{ss}$, $G_{reg}$ (resp. $\Fg_{ss}$, $\Fg_{reg}$) to denote the set of semisimple and regular semisimple elements of $G$ (resp. $\Fg$). For $x\in G_{ss}$ (resp. $X\in \Fg_{ss}$), let $Z_G(x)$ (resp. $Z_G(X)=G_X$) be the centralizer of $x$ (resp. $X$) in $G$ and let $G_x$ be the neutral component of $Z_G(x)$. Similarly, for any abelian subgroup $T$ of $G$, let $Z_G(T)$ be the centralizer of $T$ in $G$ and let $G_T$ be the neutral component of $Z_G(T)$. We say $x\in G_{reg}(\BR)$ is elliptic if $G_x(\BR)$ is a maximal elliptic torus of $G(\BR)$ (i.e. $G_x(\BR)/Z_G(\BR)$ is compact). We use $G_{ell}(\BR)$ to denote the set of regular semisimple elliptic elements of $G(\BR)$. Finally, for $x\in G_{ss}(F)$ (resp. $X\in \Fg_{ss}(F)$), let $D^G(x)$ (resp. $D^G(X)$) be the Weyl determinant.

Fix a non-degenerate, symmetric, $G$-invariant bilinear form $<\;,\;>$ (i.e. the Killing form) on $\Fg$. For any complex valued Schwartz function $f$ on $\Fg(F)$, we define its Fourier transform $\hat{f}$ (which is also a Schwartz function on $\Fg(F)$) to be
$$\hat{f}(X)=\int_{\Fg(F)} f(Y) \psi(<X,Y>) dY$$
where $dY$ is the selfdual Haar measure on $\Fg(F)$ such that $\hat{\hat{f}}(X)=f(-X)$.

We say a subset $\Omega\subset G(F)$ (resp. $\omega\subset \Fg(F)$) is $G$-invariant if it is invariant under the $G(F)$-conjugation. For any subset $\Omega\subset G(F)$ (resp. $\omega\subset \Fg(F)$), we define the $G$-invariant subset
$$\Omega^G:=\{g^{-1}\gamma g\mid g\in G(F),\gamma\in \Omega\},\;\omega^G:=\{g^{-1}\gamma g\mid g\in G(F),\gamma\in \omega\}.$$
We say a $G$-invariant subset $\Omega$ of $G(F)$ (resp. $\omega$ of $\Fg(F)$) is compact modulo conjugation if there exist a compact subset $\Gamma$ of $G(F)$ (resp. $\Fg(F)$) such that $\Omega\subset \Gamma^G$ (resp. $\omega\subset \Gamma^G$). A $G$-domain on $G(F)$ (resp. $\Fg(F)$) is an open subset of $G(F)$ (resp. $\Fg(F)$) invariant under the $G(F)$-conjugation.

Finally, we fix a minimal Levi subgroup (resp. parabolic subgroup) $M_0(F)$ (resp. $P_0(F)=M_0(F)N_0(F)$) of $G(F)$. We say a parabolic subgroup of $G(F)$ is standard if it contains $P_0(F)$. We say a Levi subgroup of $G(F)$ is standard if it is a Levi subgroup of a standard parabolic subgroup and it contains $M_0(F)$. For two Levi subgroups $L_1(F)$ and $L_2(F)$ of $G(F)$, we say that $L_1(F)$ contains $L_2(F)$ up to conjugation if there exists $g\in G(F)$ such that $L_2(F)\subset gL_1(F)g^{-1}$.

\subsection{Useful function spaces}
We use $C_{c}^{\infty}(G(F))$ to denote the space of smooth compactly supported functions on $G(F)$, and we use $\CC(G(F))$ to denote the Harish-Chandra-Schwartz space of $G(F)$ (see Section 1.5 of \cite{B15} for details). On the Lie algebra level, let $C_{c}^{\infty}(\Fg(F))$ (resp. $\CS(\Fg(F))$) be the space of smooth compactly supported functions (resp. Schwartz functions) on $\Fg(F)$. When $F$ is p-adic, we have $C_{c}^{\infty}(\Fg(F))=\CS(\Fg(F))$.

Let $C_{c,scusp}^{\infty}(G(F))$ be the space of strongly cuspidal functions in $C_{c}^{\infty}(G(F))$. Similarly we can define the spaces $\CC_{scusp}(G(F)),\; C_{c,scusp}^{\infty}(\Fg(F)),\; \CS_{scusp}(\Fg(F))$. We refer the readers to Section 5 of \cite{B15} for the definition and basic properties of strongly cuspidal functions. We say a function $f\in \CC(G(F))$ is a cusp form if all the right translations of $f$ are also strongly cuspidal. We use ${}^{\circ}\CC(G(F))$ to denote the space of cusp forms on $G(F)$.

Finally, we can also define the above function spaces with central character. For a given unitary character $\chi$ of $Z_G(F)$, let $C_{c}^{\infty}(G(F),\chi)$ be the Mellin transform of the space $C_{c}^{\infty}(G(F))$ with respect to $\chi$. Similarly, we can also define the spaces $\CC(G(F),\chi),\;C_{c,scusp}^{\infty}(G(F),\chi),$ $\CC_{scusp}(G(F),\chi),\; {}^{\circ}\CC(G(F),\chi)$.

\subsection{Representations}
When $F$ is p-adic, we say a representation $\pi$ of $G(F)$ is smooth if for every $v\in \pi$, the function
$$f:\;G(F)\rightarrow \pi,\; f(g)=\pi(g)v$$
is locally constant. When $F$ is archimedean, we say a representation $\pi$ of $G(F)$ is irreducible smooth (resp. finite length smooth) if it is an irreducible (resp. finite length) Casselman-Wallach representation of $G(F)$. We say a finite length smooth representation $\pi$ of $G(F)$ is an induced representation if there exists a proper parabolic subgroup $P=MN$ of $G$ and a finite length smooth representation $\tau$ of $M(F)$ such that $\pi=I_{P}^{G}(\tau)$. Here $I_{P}^{G}(\cdot)$ is the normalized parabolic induction.

We use $\CR(G)$ to denote the Grothendieck group of finite length smooth representations of $G(F)$, and we use $\CR(G)_{ind}\subset \CR(G)$ to denote the subspace of $\CR(G)$ generated by induced representations. The following proposition will be used in the proof of Theorem \ref{main K-type}.

\begin{prop}\label{GL(n) induction}
Assume that $F=\BR$. If $G=\GL_n$ with $n>2$ or $G=Res_{\BC/\BR}H$ where $H$ is a connected reductive group defined over $\BR$ that is not abelian, then $\CR(G)=\CR(G)_{ind}$. In other words, $\CR(G)$ is generated by induced representations.
\end{prop}

\begin{proof}
This follows from the fact that $G_{ell}(\BR)=\emptyset$ when $G=\GL_n$ ($n>2$) or when $G=Res_{\BC/\BR}H$ where $H$ is a connected reductive group defined over $\BR$ that is not abelian. More specifically, since $G_{ell}(\BR)=\emptyset$, $G(\BR)$ does not have elliptic representation. This implies that all the tempered representations of $G(\BR)$ are generated by induced representations. Together with the Langlands classification, we know that $\CR(G)=\CR(G)_{ind}$.
\end{proof}

\subsection{Quasi character and germ expansion}\label{section shalika germ}
Let $Nil(\Fg(F))$ (resp. $Nil_{reg}(\Fg(F))$) be the set of nilpotent orbits (resp. regular nilpotent orbits) of $\Fg(F)$. In particular, the set $Nil_{reg}(\Fg(F))$ is empty unless $G(F)$ is quasi-split. For every $\CO\in Nil(\Fg(F))$ and $f\in \CS(\Fg(F))$, we use $J_{\CO}(f)$ to denote the nilpotent orbital integral of $f$ associated to $\CO$. Harish-Chandra showed that there exists a unique smooth function $Y\rightarrow \hat{j}(\CO,Y)$ on $\Fg_{reg}(F)$,  which is invariant under $G(F)$-conjugation, and locally integrable on $\Fg(F)$, such that for every $f\in \CS(\Fg(F))$, we have
$$J_{\CO}(\hat{f})=\int_{\Fg(F)} f(Y) \hat{j}(\CO,Y) dY.$$

On the other hand, for $X\in \Fg_{reg}(F)$ and $f\in \CS(\Fg(F))$, let $J_G(X,f)$ be the orbital integral. Harish-Chandra also showed that there exists a unique smooth function $Y\rightarrow \hat{j}(X,Y)$ on $\Fg_{reg}(F)$,  which is invariant under $G(F)$-conjugation, and locally integrable on $\Fg(F)$, such that for every $f\in \CS(\Fg(F))$, we have
$$J_G(X,\hat{f})=\int_{\Fg(F)} f(Y) \hat{j}(X,Y) dY.$$

Assume that $F$ is p-adic. If $\theta$ is a smooth function on $G_{reg}(F)$, invariant under $G(F)-$conjugation. We say it is a quasi-character if for every $x\in G_{ss}(F)$, there is a good neighborhood $\omega_x$ of $0$ in $\Fg_x(F)$, and for every $\CO\in Nil(\Fg_x(F))$, there exists $c_{\theta,\CO}(x)\in \BC$ such that
$$\theta(x\exp(X))=\sum_{\CO\in Nil(\Fg_x(F))} c_{\theta,\CO}(x) \hat{j}(\CO,X)$$
for every $X\in \omega_{x,reg}$. We refer the readers to Section 3 of \cite{W10} for the definition of good neighborhood. The coefficients $\{c_{\theta,\CO}(x)|\; \CO\in Nil(\Fg_x(F))\}$ (resp. $\{c_{\theta,\CO}(x)|\; \CO\in Nil_{reg}(\Fg_x(F))\}$) are called the germs (resp. regular germs) of $\theta$ at $x$.

Similarly, if $\theta$ is a smooth function on $\Fg_{reg}(F)$, invariant under $G(F)-$conjugation. We say it is a quasi-character on $\Fg(F)$ if for every $X\in \Fg_{ss}(F)$, there exists an open $G_X$-invariant neighborhood $\omega_X\subset \Fg_X(F)$ of $0$, and for every $\CO \in Nil(\Fg_X(F))$, there exists $c_{\theta,\CO}(X)\in \BC$ such that
$$\theta(X+Y)=\sum_{\CO\in Nil(\Fg_X(F))} c_{\theta,\CO}(X) \hat{j}(\CO,Y)$$
for every $Y\in \omega_{X,reg}$.

When $F$ is archimedean, we refer the readers to Section 4.2-4.4 of \cite{B15} for the definition of quasi-character. In this case, the germ expansions become
$$D^G(x\exp(X))^{1/2}\theta(x\exp(X))=D^G(x\exp(X))^{1/2}\sum_{\CO\in Nil_{reg}(\Fg_x(F))} c_{\theta,\CO}(x) \hat{j}(\CO,X)+O(|X|),$$
$$D^G(X+Y)^{1/2}\theta(X+Y)=D^G(X+Y)^{1/2}\sum_{\CO\in Nil_{reg}(\Fg_X(F))} c_{\theta,\CO}(X) \hat{j}(\CO,Y)+O(|Y|).$$

The most important example of quasi-character on $G(F)$ is the Harish-Chandra character $\theta_{\pi}$ of finite length smooth representation of $G(F)$. Examples of quasi-character on $\Fg(F)$ are the functions $\hat{j}(X,\cdot)$ ($X\in \Fg_{reg}(F)$) and $\hat{j}(\CO,\cdot)$ ($\CO\in Nil(\Fg(F))$) defined above.

For $X\in \Fg_{reg}(F)$, we use $\Gamma_{\CO}(X)$ ($\CO\in Nil(\Fg(F))$ in the p-adic case and $\CO\in Nil_{reg}(\Fg(F))$ in the archimedean case) to denote the germs of the quasi-character $\hat{j}(X,\cdot)$ at $0\in \Fg(F)$. This is called the Shalika germ. In particular, we have the germ expansion
$$\hat{j}(X,Y)=\sum_{\CO\in Nil(\Fg(F))} \Gamma_{\CO}(X)\hat{j}(\CO,Y),\; F\; \text{p-adic};$$
$$D^G(X+Y)^{1/2}\hat{j}(X,Y)=D^G(X+Y)^{1/2}\sum_{\CO\in Nil_{reg}(\Fg(F))} \Gamma_{\CO}(X)\hat{j}(\CO,Y) +O(|Y|),\; F\; \text{archimedean}$$
for $Y\in \Fg_{reg}(F)$ close to $0$.

Finally, for $f\in \CC_{scusp}(G(F))$ (resp. $f\in \CS_{scusp}(\Fg(F))$), let $\theta_f$ be the quasi-character on $G(F)$ (resp. $\Fg(F)$)
defined via the weighted orbital integrals of $f$. Also for $f\in \CS_{scusp}(\Fg(F))$, let $\hat{\theta}_f=\theta_{\hat{f}}$ be the Fourier transform of $\theta_f$. We refer the readers to Section 5.2 and 5.6 of \cite{B15} for details.

\subsection{Regular germs under parabolic induction}\label{section germ parabolic}
Let $\pi$ be a finite length smooth representation of $G(F)$ and let $\theta_{\pi}$ be its Harish-Chandra character. For $x\in G_{ss}(F)$, define
$$c_{\pi}(x)=\begin{array}{cc}\left\{ \begin{array}{ccl} \frac{1}{|Nil_{reg}(\Fg_x(\BR))|} \sum_{\CO\in Nil_{reg}(\Fg_x(\BR))} c_{\theta_{\pi},\CO}(x) & \text{if} & Nil_{reg}(\Fg_x(\BR))\neq \emptyset \iff  G_x(\BR)\; \text{is quasi-split}; \\ 0 & \text{if} & Nil_{reg}(\Fg_x(\BR))=\emptyset. \\ \end{array}\right. \end{array}.$$

\begin{rmk}
\begin{enumerate}
\item \emph{For $x\in G_{reg}(\BR)$, $c_{\pi}(x)$ is just $\theta_{\pi}(x)$.}
\item \emph{If $Nil_{reg}(\Fg_x(\BR))$ only contains a unique element $\CO_x$, then $c_{\pi}(x)=c_{\theta_{\pi},\CO_x}(x)$.}
\end{enumerate}
\end{rmk}

Let $P=MN$ be a parabolic subgroup of $G$, $\tau$ be a finite length irreducible smooth representation of $M(\BR)$ and $\pi=I_{P}^{G}(\tau)$ be the normalized parabolic induction. For all $x\in G_{ss}(\BR)$, let $\CX_M(x)$ be a set of representatives for the $M(\BR)$-conjugacy classes of elements in $M(\BR)$ that are $G(\BR)$-conjugated to $x$. The following proposition was proved in Proposition 4.7.1 of \cite{B15} and it tells us the behavior of $c_{\pi}(x)$ under parabolic induction.

\begin{prop}\label{germ parabolic induction}
For all $x\in G_{ss}(\BR)$, we have
$$D^G(x)^{1/2}c_{\pi}(x) = |Z_G(x)(\BR):G_x(\BR)| \sum_{y\in \CX_M(x)} |Z_M(y)(\BR):M_y(\BR)|^{-1} D^M(y)^{1/2} c_{\tau}(y).$$
In particular, $c_{\pi}(x)=0$ if the set $\CX_M(x)$ is empty.
\end{prop}

\begin{rmk}
\emph{When $G=\GL_n$ or when $x\in G_{reg}(\BR)$, the numbers $|Z_G(x)(\BR):G_x(\BR)|$ and $|Z_M(y)(\BR):M_y(\BR)|$ are always equal to 1. Hence the equation above becomes}
$$D^G(x)^{1/2}c_{\pi}(x) = \sum_{y\in \CX_M(x)}  D^M(y)^{1/2} c_{\tau}(y).$$
\end{rmk}

\subsection{Spherical subgroup}\label{section spherical subgroup}
Let $H\subset G$ be a connected closed subgroup also defined over $F$.  We say that $H$ is a spherical subgroup if there exists a Borel subgroup $B$ of $G$ (not necessarily defined over $F$ since $G(F)$ may not be quasi-split) such that $BH$ is Zariski open in $G$. Such a Borel subgroup is unique up to $H(\bar{F})$-conjugation. If this is the case, then we say $(G,H)$ is a spherical pair and $X=G/H$ is a spherical variety of $G$.

From now on, we assume that $H$ is a spherical subgroup. We say the spherical pair $(G,H)$ is {\it minimal} if the stabilizer of the open Borel orbit is finite modulo the center. In other words, $B\cap H/Z_G\cap H$ is finite for all Borel subgroups $B\subset G$ with $BH$ open in $G$. Examples of minimal spherical varieties are the Whittaker model, the Gan-Gross-Prasad model, the Ginzburg-Rallis model, and all the split symmetric spaces. The following lemma follows from the definition of minimal spherical pair.

\begin{lem}\label{lem minimal spherical var}
Assume that $(G,H)$ is a spherical pair and $B\subset G$ be a Borel subgroup. Then $\dim(H)-\dim(Z_G\cap H)\geq \dim(G)-\dim(B)$. Moreover, the equality holds if and only if $(G,H)$ is minimal. In other words, $(G,H)$ is minimal if and only if the dimension of $H$ is equal to the dimension of the maximal unipotent subgroup of $G$ (up to modulo the center).
\end{lem}

\begin{defn}
Let $P=MN$ be a proper parabolic subgroup of $G$. For a character $\xi:N(F)\rightarrow \BC^{\times}$ of $N(F)$, we use $M_{\xi}$ to denote the neutral component of the stabilizer of $\xi$ in $M$ (under the adjoint action). For $m\in M(F)$, let ${}^m\xi$ be the character of $N(F)$ defined by ${}^m\xi(n)=\xi(m^{-1}nm)$.

We say $\xi$ is a generic character if $\dim(M_{\xi})$ is minimal, i.e. $\dim(M_{\xi})\leq \dim(M_{\xi'})$ for any character $\xi':N(F)\rightarrow \BC^{\times}$ of $N(F)$. It is easy to see that if $\xi$ is a generic character, so is ${}^m\xi$ for all $m\in M(F)$. Moreover, there are finitely many generic characters of $N(F)$ up to $M(F)$-conjugation (which is in bijection with the open $M(F)$-orbits in $\Fn(F)/[\Fn(F),\Fn(F)]$ under the adjoint action).
\end{defn}

In this paper, we restrict ourselves to the same setting as in \cite{SV}. In other words, we consider two types of spherical varieties.
\begin{itemize}
\item The reductive case, i.e. $H$ is reductive.
\item The Whittaker induction of the reductive case: there exists a parabolic subgroup $P=MN$ of $G$, and a generic character $\xi:N(F)\rightarrow \BC^{\times}$ such that $H=H_0\ltimes N$ where $H_0=M_{\xi}\subset M$ is the neutral component of the stabilizer of $\xi$ in $M$ and $H_0$ is a reductive spherical subgroup of $M$.

    In this case, we let $G_0=M$ and we say that $(G,H)$ is the Whittaker induction of $(G_0,H_0,\xi)$. If $H$ is already reductive, we just let $(G_0,H_0,\xi)=(G,H,1)$. It is easy to see that $(G,H)$ is minimal if and only if $(G_0,H_0)$ is.
\end{itemize}

\begin{rmk}
\emph{In general the stabilizer of a generic character is not necessarily reductive (e.g. the parabolic subgroup of $\GL_3$ whose Levi subgroup is $\GL_2\times \GL_1$) and also not necessarily a spherical subgroup of $M$ (e.g. the parabolic subgroup of $\GL_9$ whose Levi subgroup is $\GL_3\times \GL_3\times \GL_3$).}
\end{rmk}

We use $W_G$ to denote the Weyl group of $G(\bar{F})$. When $H$ is reductive, we use $W_X$ to denote the little Weyl group of the spherical variety $X=G/H$ (defined in \cite{Knop90}) which can be identified as a subgroup of $W_G$. Finally, let $Z_{G,H}=Z_G\cap H$ and $A_{G,H}(F)$ be the maximal split torus of $Z_{G,H}(F)$.

\section{The support of geometric multiplicity}\label{Section torus}
In this section, let $(G,H)$ be a spherical pair which is the Whittaker induction of the reductive spherical pair $(G_0,H_0,\xi)$. Recall that when $H$ is reductive, we let $(G_0,H_0,\xi)=(G,H,1)$. We are going to define a subset of semisimple conjugacy classes of $H_0(F)$, which will be the support of the geometric multiplicity.

\begin{defn}\label{defn support}
Let $\CT(G,H)$ be the set of all the closed (not necessarily connected) abelian subgroups $T(F)$ of $H_0(F)$ (up to $H_0(F)$-conjugation) satisfies the following four conditions.
\begin{enumerate}
\item Every element of $T(F)$ is semisimple and $(G_T,H_T)$ is a minimal spherical variety with $G_T(F)$ quasi-split.
\item $T(F)=Z_{Z_G(T)}(F)\cap H(F)$ where $Z_{Z_G(T)}(F)$ is the center of $Z_G(T)(F)$. In particular, we have $Z_{G,H}(F)\subset T(F)$ and $A_{G,H}(F)\subset T^\circ(F)$. Here $T^{\circ}(F)$ is the neutral component of $T(F)$ which is a subtorus of $H_0(F)$.
\item $T(F)/Z_{G,H}(F)$ (or equivalently, $T^{\circ}(F)/A_{G,H}(F)$) is compact. This is equivalent to say that $H(F)\cap A_{G_T}(F)/A_{G,H}(F)$ is finite.
\item There exists $t\in T(F)$ such that $(G_t,H_t)=(G_T,H_T)$.
\end{enumerate}
Let $\CT(G,H)^\circ=\{T(F)\in \CT(G,H)|\; T(F)=T^\circ(F) Z_{G,H}(F)\}$.
\end{defn}

\noindent
For $T(F)\in \CT(G,H)$, there exists a nonempty (this follows from Definition \ref{defn support}(4)) subset $C(T,H)$ of $T(F)/T^\circ(F)$ satisfies the following two conditions:
\begin{itemize}
\item For $\gamma \in C(T,H)$, $(G_t,H_t)=(G_T,H_T)$ for almost all $t\in \gamma T^\circ(F)$.
\item For $\gamma \in T(F)/T^\circ(F)-C(T,H)$, $(G_t,H_t)\neq (G_T,H_T)$ for all $t\in \gamma T^\circ(F)$.
\end{itemize}
In particular, for $T(F)\in \CT(G,H)^\circ$, we have $(G_t,H_t)=(G_T,H_T)$ for almost all $t\in T(F)$.

\begin{defn}
For $T(F)\in \CT(G,H)$, let $T_H(F)=\cup_{\gamma \in C(T,H)} \gamma T^\circ(F)$. $\{T_H(F)|\; T(F)\in \CT(G,H)\}$ will be the support of the geometric multiplicity.
\end{defn}

\begin{rmk}
\emph{For $t\in H_{0,ss}(F)$, $(G_t,H_t)$ is the Whittaker induction of $(G_{0,t},H_{0,t},\xi)$. Hence $\CT(G,H)=\CT(G_0,H_0)$ and $T_H(F)=T_{H_0}(F)$ for all $T(F)\in \CT(G,H)=\CT(G_0,H_0)$. In other words, the geometric multiplicity of $(G,H)$ has the same support as the geometric multiplicity of $(G_0,H_0)$.}
\end{rmk}

\begin{rmk}
\emph{Here is another way to define the support of the geometric multiplicity: it is supported on all the semisimple conjugacy classes $\{h^{-1}th|\;h\in H_0(F)\}$ of $H_0(F)$ that satisfy the following two conditions.}
\begin{enumerate}
\item $(G_t,H_t)$ is a minimal spherical variety and $G_t(F)$ is quasi-split.
\item $H(F)\cap A_{G_t}(F)/A_{G,H}(F)$ is finite.
\end{enumerate}

\emph{As we mentioned in the introduction, the quasi-split condition ensures the existence of regular nilpotent orbits in $\Fg_t(F)$. By Lemma \ref{lem minimal spherical var}, the minimal spherical variety condition ensures that the homogeneous degree of the spherical variety $(G_t,H_t)$ (which is equal to $\dim(H_t)-\dim(Z_{G_t,H_t})$) is equal to the homogeneous degree of the regular germs of $\theta_{\pi}$ at $t$ (which is equal to the dimension of the maximal unipotent subgroup of $G_t$). Meanwhile, the second condition means that the geometric multiplicity is only supported on certain ``elliptic elements".}
\end{rmk}

\begin{rmk}\label{rmk support}
\emph{When the spherical variety $X=G/H$ does not have Type N spherical root (we refer the readers to Section 3.1 of \cite{SV} for the definitions of spherical root and Type N spherical root), we expect that (although we can prove it at this moment) $T(F)=T^{\circ}(F)Z_{G,H}(F)$ for all $T(F)\in \CT(G,H)$ (i.e. $\CT(G,H)=\CT(G,H)^\circ$). In other words, the geometric multiplicity is essentially supported on tori of $H_0(F)$. On the other hand, when $X=G/H$ has Type N root, the geometric multiplicity may support on some non-connected abelian subgroups of $H_0(F)$.}

\emph{For example, as we will see in Section 8, the geometric multiplicity of the model $(\GL_{n}(\BR),\SO_{n}(\BR))$ (which has Type N root when $n>2$) is supported on the set (which is not necessarily connected when $n>2$)
$$\{diag(I_{n_1}, -I_{2n_2},t)|\; t\in T(\BR)\}$$
where
$(n_1,n_2)$ runs over the set
$$I(n_1,n_2):=\{(n_1,n_2)\in \BZ_{\geq 0}|\; n-n_1-2n_2\;\text{is a nonnegative even number}\}$$
and $T(\BR)$ is a maximal elliptic torus of $\SO_{n-n_1-2n_2}(\BR)$. The multiplicity formula for this case will be proved in Section 8.}
\end{rmk}

The next three definitions will be used in Section \ref{Section nilpotent orbit}.

\begin{defn}\label{levi of spherical variety}
Let $\CL(G,H)$ be the set of standard Levi subgroups $L(F)$ of $G(F)$ satisfy the following condition.
\begin{itemize}
\item There exists $T(F)\in \CT(G,H)^\circ$ with $T(F)\neq Z_{G,H}(F)$ such that $L(F)$ is conjugated to the Levi subgroup $Z_{G}(A_T)(F)$ where $A_T(F)$ is a maximal split torus of $G_T(F)$.
\end{itemize}
\end{defn}

\begin{defn}\label{levi of semisimple}
For $t\in G_{reg}(F)$, let $T(F)=G_t(F)$, $A_T(F)$ be the maximal split subtorus of $T(F)$, and $L(t)(F)=Z_{G}(A_T)(F)$ which is a Levi subgroup of $G(F)$. In particular, $t$ is elliptic regular if and only if $L(t)=G$. Similarly we can define $L(X)(F)$ for $X\in \Fg_{reg}(F)$.
\end{defn}

\begin{defn}\label{defn null}
We say $X\in \Fg_{reg}(F)$ is null with respect to $H$ if $L(X)$ does not contain any element in $\CL(G,H)$ up to conjugation. Apparently this definition only depends on the $G(\bar{F})$-conjugacy class (i.e. stable conjugacy class) of $X$. As a result, we say a regular semisimple conjugacy class (resp. stable conjugacy class) of $\Fg(F)$ is null with respect to $H$ if every element in it is null with respect to $H$.
\end{defn}

\begin{rmk}
\emph{If $\CT(G,H)^{\circ}=\{Z_{G,H}(F)\}$ or $\emptyset$ (e.g. the Whittaker model), the set $\CL(G,H)$ is empty, which implies that every regular semisimple element in $\Fg(F)$ is null with respect to $H$.}
\end{rmk}

\section{The constant $d(G,H,F)$ for minimal spherical varieties}\label{Section constant d(G,H,F)}
In this section, assume that $(G,H)$ is a minimal spherical pair with $H$ reductive. Moreover, we assume that $G$ is quasi-split over $F$. Then we can find a Borel subgroup $B=TN\subset G$ defined over $F$ such that $BH$ is open in $G$ and $B\cap H$ is finite modulo the center.

We use $\Fg,\Fz=\Fz_{\Fg}, \Fh,\Fb,\Ft,\Fn$ to denote the Lie algebras of $G,Z_G,H,B,T,N$. By our choice of $H$ and $B$, we have
$$\Fh\cap \Fb=\Fh\cap \Fz,\; \Fg=\Fh+ \Fb.$$
Let $\Fh'=\{X\in \Fh|\; <X,Y>=0\;\text{for all}\; Y\in \Fz\cap \Fh\}$ and $\Fh^{\perp}=\{X\in \Fg|\; <X,Y>=0\;\text{for all}\; Y\in \Fh'\}$. Then we have
$$\Fh=\Fh'\oplus (\Fz\cap \Fh),\;\Fg=\Fh'\oplus \Fb,\Fg=\Fh^{\perp}\oplus \Fn.$$
In particular, for every $t\in \Ft$, there exists unique $n_t\in \Fn$ such that $t+n_t\in \Fh^{\perp}$. By using $\Fh^{\perp}\oplus \Fn=\Fg$ again, we know that the set $\{t+n_t|\; t\in \Ft\}$ is a vector subspace of $\Fb$ of dimension $\dim(\Ft)$. We use $\Ft_H$ to denote it. It is easy to see that $\Ft_H=\Fb\cap \Fh^{\perp}$ (hence it does not depends on the choice of $T$).

\begin{lem}\label{H intersect B}
If $\Ft_{reg}\cap \Ft_H\neq \emptyset$, then $H\cap B\subset T$. In particular, $H\cap B$ is abelian.
\end{lem}

\begin{proof}
Fix $t\in \Ft_{reg}\cap \Ft_H$. Let $\gamma\in H\cap B$. In order to show that $\gamma\in T$, it is enough to show that $\gamma$ commutes with $t$. Since $\gamma\in  B$, we know that $\gamma t\gamma^{-1}=  t+n$ for some $n\in \Fn$. Since $\gamma \in H$ and $t\in \Fh^{\perp}$, we know that $t+n=\gamma t\gamma^{-1}\in \Fh^{\perp}$. This implies that $n=0$. Hence $\gamma$ commutes with $t$. This proves the lemma.
\end{proof}

\begin{defn}
Let $c(G,H,F)$ be the number of connected components of $B(F)\cap H(F)$.
\end{defn}

\begin{lem}
The number $c(G,H,F)$ is independent of the choice of $B$.
\end{lem}

\begin{proof}
Let $B=TN$ and $B'=T'N'$ be two Borel subgroups of $G$ defined over $F$ with $BH$ and $B'H$ being Zariski open in $G$. In order to prove the lemma, it is enough to show that the group $B(F)\cap H(F)$ is isomorphic to the group $B'(F)\cap H(F)$.

By Lemma \ref{H intersect B}, up to conjugating $T$ (resp. $T'$) by an element of $N(F)$ (resp. $N'(F)$), we may assume that $B\cap H\subset T$ (resp. $B'\cap H\subset T'$). Since $BH$ and $B'H$ are Zariski open in $G$, there exists $h\in H(\bar{F})$ such that $B=h^{-1}B'h$. Then the morphism
$$t\in B'\cap H \rightarrow h^{-1}th\in B\cap H$$
is an isomorphism. So it is enough to show that for all $t\in B'(F)\cap H(F)$, we have $h^{-1}th\in B(F)\cap H(F)$.

For $\sigma\in Gal(\bar{F}/F)$, since both $B$ and $B'$ are defined over $F$, we have $h^{-1}B'h=B=\sigma(h)^{-1}B'\sigma(h)$. This implies that $B'=h\sigma(h)^{-1}B'\sigma(h) h^{-1}$. Hence $h\sigma(h)^{-1}\in B'\cap H'\subset T'$. Together with the fact that $B'(F)\cap H(F)\subset T'(F)$, we have
$$\sigma(h^{-1}th)=\sigma(h)^{-1}t\sigma(h)=h^{-1}(h\sigma(h)^{-1}t\sigma(h)h^{-1})h=h^{-1}th$$
for all $t\in B'(F)\cap H(F)$. This implies that $h^{-1}th\in B(F)\cap H(F)$.
\end{proof}

\begin{lem}
There is a bijection between open orbits in $B(F)\back G(F)/H(F)$ and $\ker(H^{1}(F,H\cap B)\rightarrow H^1(F,H))$. We use $d(G,H,F)'$ to denote the number of open orbits in $B(F)\back G(F)/H(F)$.
\end{lem}

\begin{proof}
Let $X=BH$ which is an open subvariety of $G$. Then open orbits in $B(F)\back G(F)/H(F)$ are just the orbits in $B(F)\back X(F)/H(F)$. Let $B(F)\back X(F)/H(F)=\cup_{i=1}^{l} B(F)\gamma_i H(F)$. For each $i$, there exists $b_i\in B(\bar{F})$ and $h_i\in H(\bar{F})$ such that $\gamma_i=b_ih_i$. Then it is easy to see that the map
$$\sigma\in Gal(\bar{F}/F)\mapsto  b_{i}^{-1}\sigma(b_i)=h_i \sigma(h_i)^{-1}\in H\cap B$$
is a cocycle whose image in $H^1(F,H\cap B)$ only depends on the orbit $B(F)\gamma_i H(F)$. Also by definition, this cocycle becomes a coboundary in $H$. This gives a well defined map from $B(F)\back X(F)/H(F)$ to $\ker(H^{1}(F,H\cap B)\rightarrow H^1(F,H))$. One can easily check that this map is a bijection.
\end{proof}

\begin{defn}
We define the constant $d(G,H,F)$ to be
$$d(G,H,F)=d'(G,H,F) \times \frac{|W_G|}{|W_X|}.$$
Recall that $W_X$ is the little Weyl group of the spherical variety $X=G/H$ and $W_G$ is the Weyl group of $G(\bar{F})$.
\end{defn}

\begin{rmk}
\emph{Since $(G,H)$ is a minimal spherical pair, it is wavefront if and only if $W_G=W_X$. If this is the case, we have
$$d(G,H,F)=d'(G,H,F)=|\ker(H^{1}(F,H\cap B)\rightarrow H^1(F,H))|.$$
We refer the readers to Section 2.1 of \cite{SV} for the definition of wavefront spherical variety.}
\end{rmk}

The rest of this subsection is to study the relation between the number $d(G,H,F)$ and the slice representation (i.e. the conjugation action of $H(F)$ on $\Fh^{\perp}(F)$).
\begin{lem}
There exists a $W_G$-invariant Zariski open subset $\Ft^0$ of $\Ft_{reg}$ such that for all $t\in \Ft^0(\bar{F})$, the $G(\bar{F})$-conjugacy class of $t$ in $\Fh^{\perp}(\bar{F})$ breaks into $\frac{|W_G|}{|W_X|}$-many $H(\bar{F})$-conjugacy classes.
\end{lem}

\begin{proof}
By modulo $H$ and $G$ by the center $Z_{G,H}=H\cap Z_G$, we may assume that $H\cap Z_G=\{1\}$. Then we know that $B\cap H$ is finite.
We denote by $\CX(T)$ the group of rational characters of $T$, and define $\Fa=\Hom(\CX(T),\BR)$. Let $\CX(X)$ be the group of $T$-eigencharacters on $\bar{F}(X)^{(B)}$ where $\bar{F}(X)^{(B)}$ is the multiplicative group of nonzero $B$-eigenfunctions on $\bar{F}(X)$. Finally, let $\Fa_X=\Hom(\CX(X),\BR)$. Since $H\cap B$ is finite, we have $\Fa=\Fa_X$. Let $\Fa^{\ast}=\Fa_{X}^{\ast}$ be the dual of $\Fa=\Fa_X$, and let $T^{\ast}X=\Fh^{\perp}\times_H G$ be the cotangent bundle of $X$. By the result in \cite{Knop90}, we have $\Fh^{\perp}\sslash H= T^{\ast}X \sslash G=\Fa_{X}^{\ast} \sslash W_X=\Fa^{\ast} \sslash W_X$. Meanwhile, we have $\Fg\sslash G=\Fa^{\ast} \sslash W_G$. This proves the lemma.
\end{proof}

\begin{rmk}
\emph{When $(G,H)$ is a symmetric pair (which is wavefront), we have $W_G=W_X$. By the work of Kostant-Rallis \cite{KR}, we can even take $\Ft^0$ to be $\Ft_{reg}$. Examples of non wavefront minimal spherical varieties are $(\SO_{2n+1},\GL_n)$ and $(\GL_{2n+1},\Sp_{2n})$.}
\end{rmk}

\begin{defn}
Let $\Fh^{\perp,0}$ be the set of elements in $\Fh^{\perp}$ that is $G$-conjugated to an element in $\Ft^{0}$. It is a Zariski open subset of $\Fh^{\perp}$. By the above lemma, we know that each $G(\bar{F})$-conjugacy class in $\Fh^{\perp,0}(\bar{F})$ breaks into $\frac{|W_G|}{|W_X|}$-many $H(\bar{F})$-conjugacy classes.
\end{defn}

\begin{lem}
For every $t\in \Ft_H(F)$ regular semisimple, the $H(\bar{F})$-conjugacy class of $t$ in $\Fh^{\perp}(F)$ breaks into $d(G,H,F)'$ many $H(F)$-conjugacy classes.
\end{lem}

\begin{proof}
By conjugating $T$ we may assume that $t\in \Ft_{reg}(F)$. By Lemma \ref{H intersect B}, we know that $H\cap B\subset T$. Let $t'\in \Fh^{\perp}(F)$ be an element that is $H(\bar{F})$-conjugated to $t$. Then exists $h\in H(\bar{F})$ such that $ht'h^{-1}=t$. For all $\sigma\in Gal(\bar{F}/F)$, we have
$$\sigma(h)t'\sigma(h)^{-1}=ht'h^{-1}=t.$$
In particular, $\sigma(h)h^{-1}$ commutes with $t$. This implies that $\sigma(h)h^{-1}\in H\cap T=H\cap B$. Then it is easy to see that the map
$$\sigma\in Gal(\bar{F}/F)\mapsto  \sigma(h)h^{-1}\in H\cap B$$
is a cocycle whose image in $H^1(F,H\cap B)$ only depends on the $H(F)$-conjugacy classes of $t'$. Also it is easy to see that this cocycle becomes a coboundary in $H$. This gives a well defined map from the set of $H(F)$-conjugacy classes in the $H(\bar{F})$-conjugacy class of $t$ in $\Fh^{\perp}(F)$ to $\ker(H^{1}(F,T_0)\rightarrow H^1(F,H))$. One can easily check that this map is a bijection.
\end{proof}

Combining the lemmas above, we have proved the following proposition.

\begin{prop}\label{quasi-split conjugacy classes}
For every $t\in \Fh^{\perp,0}(F)$, if $G_t(F)$ is a maximal quasi-split torus of $G(F)$ (i.e. the conjugacy class of $t$ is ``quasi-split"), then the $G(\bar{F})$-conjugacy class of $t$ (i.e. the stable conjugacy class of $t$) in $\Fh^{\perp}(F)$ breaks into $d(G,H,F)=d(G,H,F)'\times \frac{|W_G|}{|W_X|}$ many $H(F)$-conjugacy classes.
\end{prop}

\begin{rmk}
\emph{If $H\cap B\subset Z_G$, then by the same argument as above, we can even show that every $G(\bar{F})$-conjugacy class (not necessarily quasi-split) in $\Fh^{\perp,0}(F)$ breaks into $d(G,H,F)$ many $H(F)$-conjugacy classes.}
\end{rmk}

\begin{rmk}\label{rmk stable}
\emph{In general, if $(G,H)$ is the Whittaker induction of $(G_0,H_0,\xi)$ with $(G_0,H_0)$ minimal, we can also define an analogue of space $\Fh^{\perp}(F)$ by adding the information of $\xi$ (see Section \ref{Section slice nonred}). We will denote this space by $\Xi+\Fh_{0}^{\perp}(F)+\Fn(F)$ and we are still interested in how the stable conjugacy classes in $\Xi+\Fh_{0}^{\perp}(F)+\Fn(F)$ decomposes into $H(F)$-conjugacy classes.}

\emph{For most known cases, the stable conjugacy classes in $\Xi+\Fh_{0}^{\perp}(F)+\Fn(F)$ are the same as the $H(F)$-conjugacy classes, i.e. $d(G_0,H_0,F)=1$. In other words, two regular semisimple elements in $\Xi+\Fh_{0}^{\perp}(F)+\Fn(F)$ are $G(\bar{F})$-conjugated to each other if and only if they are $H(F)$-conjugated to each other. For the Whittaker model case, this follows from the theory of Kostant section \cite{Kos}. For the Gan-Gross-Prasad model case, this was proved in Section 9 of \cite{W10} (the orthogonal case) and Section 10 of \cite{B15} (unitary case). For the Ginzburg-Rallis model case, this was proved in Section 8 of \cite{Wan15}. This property is crucial in the proof of the local trace formula for those cases.}

\emph{The only exception among the known cases is the Ginzburg-Rallis model for unitary group (see Section \ref{section GR}). In that case, the number $d(G_0,H_0,F)$ is equal to 2 which means that every $G(\bar{F})$-conjugacy class in $\Xi+\Fh_{0}^{\perp}(F)+\Fn(F)$ breaks into two $H(F)$-conjugacy classes. However, although we have proved the multiplicity formula for this model in \cite{WZ}, it was not proved by the trace formula method. Instead, we first considered the Ginzburg-Rallis model for unitary similitude group (where the number $d(G_0,H_0,F)$ is equal to 1). We proved the trace formula and the multiplicity formula for the unitary similitude group case. Then we proved the multiplicity formula for the unitary group case by using the multiplicity formula of the unitary similitude group case.}

\emph{Hence if one wants to prove the multiplicity formula and local trace formula for general spherical varieties, one of the important steps is to develop a method to deal with the case when $d(G_0,H_0,F)\neq 1$. Roughly speaking, we need to ``stabilize" the trace formula.}
\end{rmk}

\section{Nilpotent orbits associated to minimal spherical varieties}\label{Section nilpotent orbit}
In this subsection, let $(G,H)$ be a minimal spherical pair with $G(F)$ quasi-split. The goal is to define a subset $\CN(G,H,\xi)$ (note that $\xi=1$ when $H$ is reductive) of $Nil_{reg}(\Fg(F))$.

\subsection{Conjugacy classes associated to regular nilpotent orbits}
Fix a regular nilpotent orbit $\CO$ of $\Fg(F)$. For $\Xi\in \CO$, by the theory of $\Fs\Fl_2$-triple, there exists a homomorphism
$$\varphi:\; F^{\times}\rightarrow G(F)$$
such that for all $s\in F^{\times}$, we have $\varphi(s)\Xi\varphi(s)^{-1}=s^{-2} \Xi$. Since $\CO$ is regular, $\varphi$ is unique up to the center (i.e. two different choices of $\varphi$ are differed by an element in $\Hom(F^{\times},Z_G(F))$). Let $N(F)$ (resp. $\bar{N}(F)$) be the unipotent subgroup of $G(F)$ whose Lie algebra is given by
$$\Fn(F)=\{X\in \Fg(F)|\; \lim_{s\rightarrow 0} \varphi(s) X\varphi(s)^{-1}=0\},\; \bar{\Fn}(F)=\{X\in \Fg(F)|\; \lim_{s\rightarrow 0} \varphi(s)^{-1} X\varphi(s)=0\}.$$
In particular, we have $\Xi\in \bar{\Fn}(F)$. Finally, let $T(F)$ be the centralizer of $Im(\varphi)$ in $G(F)$. Since $\CO$ is regular, we know that $N(F)$ (resp. $\bar{N}(F)$) is a maximal unipotent subgroups of $G(F)$, $T(F)$ is a maximal torus of $G(F)$, $B=T(F)N(F)$ (resp. $\bar{B}(F)=T(F)\bar{N}(F)$) is a Borel subgroup of $G(F)$, $B(F)$ and $\bar{B}(F)$ are opposite to each other.

\begin{rmk}
\emph{The map
$$\xi: N(F)\rightarrow \BC^{\times},\;\; \xi(\exp(X))=\psi(<\Xi,X>),\;X\in \Fn(F)$$
is a generic character of $N(F)$.}
\end{rmk}

\begin{defn}
For $X\in \Fg_{reg}(F)$, we say that $X$ is associated to $\CO$ if $X$ is $G(F)$-conjugated to an element in $\Xi+\Fb(F)$. We say a regular semisimple conjugacy class of $\Fg(F)$ is associated to $\CO$ if all the elements in this conjugacy class are associated to $\CO$. It is easy to see that this definition does not depend on the choice of $\Xi$. $\Xi+\Fb(F)$ is called the Kostant section associated to $\CO$.
\end{defn}

\begin{rmk}
\emph{By the theory of Kostant section \cite{Kos}, for every stable regular semisimple conjugacy class of $\Fg(F)$, there is a unique conjugacy class inside it that is associated to $\CO$. Later in Section \ref{section Whittaker}, we will show that for two different regular nilpotent orbits $\CO_1,\CO_2\in Nil_{reg}(\Fg(F))$, there exists a regular semisimple conjugacy class of $\Fg(F)$ that is associated to $\CO_1$, but not associated to $\CO_2$.}
\end{rmk}

\begin{lem}\label{Shalika germ}
When $F$ is p-adic, for all regular semisimple conjugacy classes $\{gXg^{-1}|\;g\in G(F)\}$ of $\Fg(F)$, $\Gamma_{\CO}(X)=1$ if and only if $X$ is associated to $\CO$. Here $\Gamma_{\CO}(X)$ is the Shalika germ defined in Section \ref{section shalika germ}.
\end{lem}

\begin{proof}
This was proved by Kottwitz in \cite{K}. See Proposition 4.2 of \cite{DR10} for a different proof.
\end{proof}

\begin{rmk}
\emph{In general we expect the above lemma also holds when $F=\BR$ (the case when $F=\BC$ is trivial).}
\end{rmk}

\subsection{The reductive case}\label{Section slice red}
We first consider the case when $H$ is reductive. In the previous section, we have defined the subspace $\Fh^{\perp}(F)$ of $\Fg(F)$.

\begin{defn}
Let $\CN(G,H,1)$ be the subset of $Nil_{reg}(\Fg(F))$ consisting of elements $\CO\in Nil_{reg}(\Fg(F))$ satisfy the following condition.
\begin{itemize}
\item For almost all regular semisimple conjugacy classes of $\Fg(F)$, if the conjugacy class is null with respect to $H$ and is associated to $\CO$, then this conjugacy class has nonempty intersection with $\Fh^{\perp}(F)$ (i.e. there exists $X\in \Fh^{\perp}(F)$ such that $X$ belongs to this conjuacy class).
\end{itemize}
We refer the readers to Definition \ref{defn null} for the definition of null.
\end{defn}

\subsection{The nonreductive case}\label{Section slice nonred}
Now we consider the non-reductive case. Let $(G,H)$ be the parabolic induction of $(G_0,H_0,\xi)$. In other words, there exists a parabolic subgroup of $P=MN$ of $G$, and a generic character $\xi:N(F)\rightarrow \BC^{\times}$ of $N(F)$ such that
\begin{itemize}
\item $G_0=M$ and $H=H_0\ltimes N$ where $H_0\subset G_0=M$ is the neutral component of the stabilizer of the character $\xi$.
\end{itemize}
Let $\bar{P}=M\bar{N}$ be the opposite parabolic subgroup and let $\Xi\in \bar{\Fn}(F)$ be the unique element such that
$$\xi(\exp(X))=\psi(<\Xi, X>),\;\forall X\in \Fn(F).$$

Since $(G,H)$ is minimal, so it $(G_0,H_0)$. By the discussion of the reductive case, we have the subspace $\Fh_{0}^{\perp}(F)$ of $\Fg_0(F)=\Fm(F)$.

\begin{defn}
With the notations above, let $\CN(G,H,\xi)$ be the subset of $Nil_{reg}(\Fg(F))$ consisting of elements $\CO\in Nil_{reg}(\Fg(F))$ satisfy the following condition.
\begin{itemize}
\item For almost all regular semisimple conjugacy classes of $\Fg(F)$, if the conjugacy class is null with respect to $H$ and is associated to $\CO$, then this conjugacy class has nonempty intersection with $\Xi+\Fh_{0}^{\perp}(F)+\Fn(F)$ (i.e. there exists $X\in \Fh_{0}^{\perp}(F)$ and $N\in \Fn(F)$ such that $\Xi+X+N$ belongs to this conjuacy class).
\end{itemize}
\end{defn}

\begin{rmk}
\emph{This definition depends on the generic character $\xi$.}
\end{rmk}

\begin{conj}
The set $\CN(G,H,\xi)$ is non empty.
\end{conj}

To end this section, we want to point that the notion of {\it null} is crucial in our definition of the set $\CN(G,H,\xi)$. The reason is that in most cases, the tangent space $\Fh^{\perp}(F)$ (or $\Xi+\Fh_{0}^{\perp}(F)+\Fn(F)$ in the nonreductive case) does not contain all the regular semisimple stable conjugacy classes of $\Fg(F)$, but we do expect it contains all the regular semisimple stable conjugacy classes that are null with respect to $H$. Here are some examples.

\textbf{For the model $(G(F),H(F))=(\GL_{2n}(\BR),\SO_{2n}(\BR))$}, the set $\CT(G,H)^{\circ}$ consists of subgroups of the form $\pm I_{2n-2m}\times (\BC^1)^m$ with $0\leq m\leq n$ (see Lemma \ref{support GL(n),SO(n) even}). Here $\BC^1$ is the norm one elements in $\BC^{\times}$ identified with a torus of $\GL_{2}(\BR)$ via the map $e^{i\theta}\rightarrow \begin{pmatrix}\cos \theta & \sin \theta \\ -\sin \theta & \cos \theta\end{pmatrix}$. As a result, the set $\CL(G,H)$ consists of all the standard Levi subgroups of $\GL_{2n}(\BR)$ of the form $(\GL_2(\BR))^{m}\times (\GL_1(\BR))^{2n-2m}$ for $1\leq m\leq n$. This implies that a regular semisimple conjugacy class in $\Fg(\BR)=\Fg\Fl_{2n}(\BR)$ is null with respect to $H$ if and only if all its eigenvalues are real numbers. On the other hand, from basic linear algebra, we know that the eigenvalues of symmetric real matrix are real numbers. This implies that $\Fh^{\perp}(\BR)$ only contains those conjugacy classes that are null with respect to $H$. A similar discussion also holds for the model $(G(F),H(F))=(\GL_{2n+1}(\BR),\SO_{2n+1}(\BR))$.

\textbf{For the model $(G,H)=(\GL_3,\SL_2)$}, the set $\CT(G,H)^{\circ}$ consists of all the maximal elliptic tori of $\SL_2(F)$ and the trivial torus. Hence the set $\CL(G,H)$ contains all the standard Levi subgroups of $\GL_3$ of the form $\GL_2\times \GL_1$. As a result, a regular semisimple conjugacy class in $\Fg(F)=\Fg\Fl_3(F)$ is null with respect to $H$ if and only if all the eigenvalues belong to $F$ (i.e. its centralizer in $G(F)$ is a split torus). On the other hand, it is easy to see that a regular semisimple conjugacy class appears in $\Fh^{\perp}(F)$ if and only if at least one of its eigenvalues belongs to $F$ (i.e. it is not elliptic). In particular, $\Fh^{\perp}(F)$ does not contain all the regular semisimple conjugacy classes of $\Fg(F)$, but it contains all the the regular semisimple conjugacy classes that are null with respect to $H$.

Another way to understand the notion of null is via the quasi-character $\theta=\hat{j}(X,\cdot)$ ($X\in \Fg_{reg}(F)$) on $\Fg(F)$ defined in Section \ref{section shalika germ}. By the definition of null and Proposition 4.7.1 of \cite{B15}, if $X$ is null with respect to $H$, then the regular germs of $\theta$ at $\Ft(F)$ is equal to zero for all $T(F)\in \CT(G,H)^\circ$ with $T(F)\neq Z_{G,H}(F)$. Here $\Ft(F)$ is the Lie algebra of $T^\circ(F)$.

\section{The conjectural multiplicity formula and trace formula}\label{Section multiplicity formula}

\subsection{The multiplicity formula}
Let $(G,H)$ be a spherical variety that is the parabolic induction of the reductive pair $(G_0,H_0,\xi)$ (as in the previous sections, if $(G,H)$ is reductive, we just let $(G_0,H_0,\xi)=(G,H,1)$). Let $\omega:H_0(F)\rightarrow \BC^{\times}$ be a unitary character. Then $\omega\otimes \xi$ is a character on $H(F)=H_0(F)\ltimes N(F)$. For any irreducible smooth representation $\pi$ of $G(F)$, we define the multiplicity
$$m(\pi,\omega\otimes \xi):=\dim(\Hom_{H(F)}(\pi,\omega\otimes \xi)).$$

Recall that $Z_{G,H}(F)=Z_G(F)\cap H(F)$ and $A_{G,H}(F)$ is the maximal split torus of $Z_{G,H}(F)$. Let $\eta$ be the restriction of the character $\omega$ to $A_{G,H}(F)$. Then we know that $m(\pi,\omega\otimes \xi)=0$ unless the central character of $\pi$ is equal to $\eta$ on $A_{G,H}(F)$. We fix a central character $\chi:Z_G(F)\rightarrow \BC^{\times}$ with $\chi|_{A_{G,H}(F)}=\eta$. Let $Irr(G,\chi)$ be the set of all the irreducible smooth representations of $G(F)$ whose central character is equal to $\chi$. We use $\Pi_{temp}(G,\chi)$ (resp. $\Pi_{disc}(G,\chi),\;\Pi_{cusp}(G,\chi)$) to denote the set of tempered representations (resp. discrete series, supercuspidal representations) in $Irr(G,\chi)$.

For $T(F)\in \CT(G,H)$, we have defined $T_H(F)=\cup_{\gamma\in C(T,H)} \gamma T^\circ(F)$ in Section \ref{Section torus}. Let $dt$ be the Haar measure on $T^\circ(F)/A_{G,H}(F)$ such that the total volume is 1 (note that $T^\circ(F)/A_{G,H}(F)$ is compact). This induces a measure $dt$ on $T_H(F)/A_{G,H}(F)=\cup_{\gamma\in C(T,H)} \gamma \cdot T^{\circ}(F)/A_{G,H}(F)$.

Now we are ready to define the geometric multiplicity.

\begin{defn}\label{definition geometric multiplicity}
Let $\theta$ be a quasi-character on $G(F)$ with central character $\chi$ (i.e. $\theta(zg)=\chi(z)\theta(g)$ for $z\in Z_G(F)$ and $g\in G_{reg}(F)$). Define
$$m_{geom}(\theta)=\sum_{T(F)\in \CT(G,H)} |W(H_0,T)|^{-1}\int_{T_H(F)/A_{G,H}(F)} \omega^{-1}(t)D^H(t)$$
$$\frac{d(G_{0,T},H_{0,T},F)}{|Z_{H_0}(T)(F):H_{0,T}(F)|\times c(G_{0,T},H_{0,T},F)}\times \frac{1}{|\CN(G_T,H_T,\xi)|} \sum_{\CO\in \CN(G_T,H_T,\xi)} c_{\theta,\CO}(t)dt.$$
Here $dt$ is the Haar measure on $T_H(F)/A_{G,H}(F)$ defined above, the numbers $d(G_{0,T},H_{0,T},F), \; c(G_{0,T},H_{0,T},F)$ are defined in Section \ref{Section constant d(G,H,F)}, and $W(H_0,T)=N_{H_0}(T)(F) /Z_{H_0}(T)(F)$ where $N_{H_0}(T)(F)$ is the normalizer of $T(F)$ in $H_0(F)$. Note that the number
$$\frac{1}{|Z_{H_0}(T)(F):H_{0,T}(F)|\times c(G_{0,T},H_{0,T},F)}$$
is an analogue of $\frac{1}{Z_H(x)}$ for the finite group case in \eqref{finite group case}.

Then For $\pi\in Irr(G,\chi)$, we define the geometric multiplicity
$$m_{geom}(\pi,\omega\otimes \xi)=m_{geom}(\theta_{\pi}).$$
\end{defn}

\begin{rmk}
\emph{In general, the integral defining $m_{geom}(\pi,\omega\otimes \xi)$ may not be absolutely convergent, and one would need to regularize it.}

\emph{Among all the known cases (i.e. Whittaker model, Gan-Gross-Prasad model, Ginzburg-Rallis model, Galois model, and Shalika model), the integral defining $m_{geom}(\pi,\omega\otimes \xi)$ is convergent for Whittaker model (this is trivial), orthogonal Gan-Gross-Prasad model (Proposition 7.3 of \cite{W10}), Ginzburg-Rallis model (Proposition 5.2 of \cite{Wan15}), Galois model (Section 4.1 of \cite{B18}), and Shalika model (Lemma 3.2 of \cite{BW}). For unitary Gan-Gross-Prasad model, the integral is not convergent and one needs to regularize it (Section 5 of \cite{B12} and Section 11.1 of \cite{B15}).}
\end{rmk}

\begin{defn}
When $H$ is reductive, we say $(G,H)$ is tempered (resp. strongly tempered) if all the matrix coefficients of discrete series (resp. tempered representations) of $G(F)$ are integrable on $H(F)/A_{G,H}(F)$. In general, if $(G,H)$ is the Whittaker induction of $(G_0,H_0,\xi)$, we say $(G,H)$ is tempered (resp. strongly tempered) if $(G_0,H_0)$ is tempered (resp. strongly tempered).
\end{defn}

\begin{conj}\label{main conjecture}
\begin{enumerate}
\item $m(\pi)=m_{geom}(\pi)$ for all $\pi\in \Pi_{cusp}(G,\chi)$.
\item If $(G,H)$ is tempered, then $m(\pi)=m_{geom}(\pi)$ for all $\pi\in \Pi_{disc}(G,\chi)$. Moreover, let $d\pi$ be the natural measure on the set $\Pi_{temp}(G,\chi)$ as defined in Section 2.6 of \cite{B15}. Then $m(\pi)=m_{geom}(\pi)$ for almost all $\pi\in \Pi_{temp}(G,\chi)$ (under the measure $d\pi$).
\item If $(G,H)$ is strongly tempered, then $m(\pi)=m_{geom}(\pi)$ for all $\pi\in \Pi_{temp}(G,\chi)$.
\end{enumerate}
\end{conj}

As we said in the introduction, in general, if we want the multiplicity formula holds for all irreducible smooth representations (or even finite length smooth representations) of $G(F)$, we need to replace the multiplicity by the Euler-Poincar\'e pairing. One reason is that both the Harish-Chandra character and the Euler-Poincar\'e pairing behave nicely under the short exact sequence, while the multiplicity does not.  This was first observed by Prasad in \cite{P}. To be specific, for two smooth (not necessarily finite length) representations $\pi$ and $\pi'$ of $G(F)$, we define the Euler-Poincar\'e pairing
$$\text{EP}_{G}[\pi,\pi']=\sum_{i} (-1)^i \dim(\text{Ext}_{G}^{i}[\pi,\pi']).$$
Then for a finite length smooth representation $\pi$ of $G(F)$, we define (here for simplicity we assume that the split center $A_{G,H}(F)$ is trivial)
$$\text{EP}(\pi,\omega\otimes \xi)=\text{EP}_{G}(\pi,\text{Ind}_{H}^{G}(\omega\otimes \xi)).$$

\begin{conj}\label{conj EP}
Given a finite length smooth representation $\pi$ of $G(F)$, the followings hold.
\begin{enumerate}
\item $\text{EP}(\pi,\omega\otimes \xi)$ is well defined. In other words, $\text{Ext}_{G}^{i}(\pi,\text{Ind}_{H}^{G}(\omega\otimes \xi))$ is finite dimensional for all $i\geq 0$.
\item $\text{EP}(\pi,\omega\otimes \xi)=m_{geom}(\pi,\omega\otimes \xi)$.
\end{enumerate}
\end{conj}
When $F$ is p-adic, the first part of the conjecture was proved by Aizenbud and Sayag in \cite{AS}.

\begin{rmk}
\emph{When $\pi$ is supercuspidal, we have $\text{Ext}_{G}^{i}(\pi,\text{Ind}_{H}^{G}(\omega\otimes \xi))=0$ for $i>0$, which implies that $\text{EP}(\pi,\omega\otimes \xi)=m(\pi,\omega\otimes \xi)$. This is why the multiplicity formula $m(\pi,\omega\otimes\xi)=m_{geom}(\pi,\omega\otimes \xi)$ should always hold in the supercuspidal case.}
\end{rmk}

In Section \ref{Section known cases}, we will show that Conjecture \ref{main conjecture} holds for Whittaker model, Gan-Gross-Prasad model, Ginzburg-Rallis model, Galois model and Shalika model. For each of these cases, there is a multiplicity formula that has already been proved. Hence in order to prove Conjecture \ref{main conjecture}, we just need to show that our definition of the geometric multiplicity matches the one in the known multiplicity formula. On the other hand, Conjecture \ref{conj EP} is more difficult. The only known cases are the group case $(G,H)=(H\times H,H)$, the Whittaker model, and the Gan-Gross-Prasad model for the general linear group (see Proposition 2.1, Proposition 2.8 and Theorem 4.2 of \cite{P}).

\subsection{The trace formula}
We use the same notation as in the previous subsection. We first need to define the space of test functions. When $(G,H)$ is tempered, we require $f\in \CC_{scusp}(G(F),\chi)$. When $(G,H)$ is not tempered, we require $f\in {}^\circ \CC(G(F),\chi)\cap C_{c}^{\infty}(G(F),\chi)$. For such a test function $f$, we define the distribution $I(f)$ of the trace formula to be
$$I(f)=\int_{H(F)\back G(F)} \int_{H(F)/A_{G,H}(F)} f(g^{-1}hg) \omega\otimes \xi(h)^{-1} dhdg.$$
In general the double integral above is not absolutely convergent (although each individual integral is usually convergent) and one needs to introduce some truncation functions on $H(F)\back G(F)$.

For the geometric expansion, let $\theta_f$ be the quasi-character on $G(F)$ defined via the weighted orbital integrals of $f$. We define the geometric expansion of the trace formula to be
$$I_{geom}(f)=m_{geom}(\theta_f)$$
where $m_{geom}(\theta_f)$ was defined in Definition \ref{definition geometric multiplicity}.

For the spectral expansion, when $(G,H)$ is not tempered, let
\begin{equation}\label{spectra l}
I_{spec}(f)=\sum_{\pi\in \Pi_{cusp}(G,\chi)} m(\pi,\omega\otimes \xi) \tr(\pi^{\vee}(f))
\end{equation}
where $\pi^{\vee}$ is the contragredient of $\pi$. When $(G,H)$ is tempered, let
\begin{equation}\label{spectral 2}
I_{spec}(f)=\int_{\CX(G,\chi)} D(\pi)\theta_f(\pi^{\vee}) m(\pi,\omega\otimes \xi)d\pi.
\end{equation}
Here $\CX(G,\chi)$ is a set of virtual tempered representations of $G(F)$ with central character $\chi$ defined in Section 2.7 of \cite{B15}, the number $D(\pi)$ and the measure $d\pi$ are also defined in Section 2.7 of \cite{B15}, and $\theta_f(\pi^{\vee})$ is defined in Section 5.4 of \cite{B15} via the weighted character. Now we are ready to state the conjectural trace formula.

\begin{rmk}
\emph{When $f\in {}^\circ \CC(G(F),\chi)\cap C_{c}^{\infty}(G(F),\chi)$, the expression on the right hand side of \eqref{spectral 2} is equal to the one on the right hand side of \eqref{spectra l}.}
\end{rmk}

\begin{conj}\label{conj trace formula}
\begin{enumerate}
\item When $(G,H)$ is tempered, the trace formula $I_{geom}(f)=I(f)=I_{spec}(f)$ holds for all $f\in \CC_{scusp}(G(F),\chi)$.
\item When $(G,H)$ is not tempered, the trace formula $I_{geom}(f)=I(f)=I_{spec}(f)$ holds for all $f\in {}^\circ \CC(G(F),\chi)\cap C_{c}^{\infty}(G(F),\chi)$.
\end{enumerate}
\end{conj}

Like the conjectural multiplicity formula, by our discussion in Section \ref{Section known cases}, we know that Conjectural \ref{conj trace formula} holds for Whittaker model, Gan-Gross-Prasad model, Ginzburg-Rallis model, Galois model and Shalika model.

\begin{rmk}
\emph{Although the trace formulas are the same for the tempered case and the strongly tempered case, the multiplicity formula behaves differently. As we discussed in Conjecture \ref{main conjecture}, for the strongly tempered case, the multiplicity formula should hold for all tempered representations; while for the non-strongly tempered case, it only holds for all discrete series and for almost all tempered representations. An easy example of this kind would be the Shalika model (see Remark 3.4 of \cite{BW}).}
\end{rmk}

\subsection{The case when $\omega$ is not a character}\label{section K-types}
In the subsection, assume that $F=\BR$ and $H(\BR)=K$ is a maximal connected compact subgroup of $G(\BR)$. Let $\omega$ be a finite dimensional representation of $H(\BR)$. For a finite length smooth representation $\pi$ of $G(\BR)$, we can still define the multiplicity $m(\pi,\omega)$ and the Euler-Poincar\'e pairing $\text{EP}(\pi,\omega)$ as in the previous subsections. Moreover, since $H(\BR)$ is compact, we have $m(\pi,\omega)=\text{EP}(\pi,\omega)$.

Meanwhile, let $\omega^{\vee}$ be the dual representation of $\omega$ and let
$$\theta_{\omega^{\vee}}(h)=\tr(\omega^{\vee}(h)),\;h\in H(\BR)$$
be the character of $\omega^{\vee}$. Then we can define the geometric multiplicity $m_{geom}(\pi,\omega)$ as in the character case in Definition \ref{definition geometric multiplicity} except that we replace $\omega^{-1}$ by $\theta_{\omega^{\vee}}$. To be specific, we define
$$m_{geom}(\pi,\omega)=\sum_{T(F)\in \CT(G,H)} |W(H,T)|^{-1}\int_{T_H(F)/A_{G,H}(F)} \theta_{\omega^{\vee}}(t)D^H(t)$$
$$\frac{d(G_{T},H_{T},F)}{|Z_{H}(T)(F):H_{T}(F)|\times c(G_{T},H_{T},F)}\times \frac{1}{|\CN(G_T,H_T,1)|} \sum_{\CO\in \CN(G_T,H_T,1)} c_{\theta_{\pi},\CO}(t)dt.$$

\begin{conj}\label{conj K-type}
For all finite length smooth representations $\pi$ of $G(\BR)$, we have $m(\pi,\omega)=m_{geom}(\pi,\omega)$.
\end{conj}

Conjecture \ref{conj K-type} gives a multiplicity formula of K-types for all finite length smooth representations of $G(\BR)$. In Section 8 and 9, we will prove Conjecture \ref{conj K-type} when $G(\BR)=\GL_n(\BR)$ and when $G=Res_{\BC/\BR}H$ is a complex reductive group. Apparently it is enough to prove the conjecture when $\pi$ and $\omega$ are irreducible.

\section{The known cases}\label{Section known cases}
In this section, assume that $F$ is p-adic. We will show that for all the known cases, the geometric multiplicity defined in Definition \ref{definition geometric multiplicity} matches the one in the multiplicity formula that has been proved. This would imply that Conjecture \ref{main conjecture} and \ref{conj trace formula} hold for all these cases. We consider the Wittaker model in Section \ref{section Whittaker}, the Gan-Gross-Prasad model in Section \ref{section GGP}, the Ginzburg-Rallis model in Section \ref{section GR}, the Galois model in Section \ref{section Galois}, and the Shalika model in Section \ref{section Shalika}.

We would like to point out that all the models above do not have Type N root. And for all these models, we have $\CT(G,H)=\CT(G,H)^{\circ}$ (i.e. the geometric multiplicity only supports on tori of $G(F)$). This matches the discussion in Remark \ref{rmk support}.

\subsection{The Whittaker model}\label{section Whittaker}
Let $G$ be a connected reductive group defined over $F$. Assume that $G(F)$ is quasi-split. Let $B=TN$ be a Borel subgroup of $G$, $\bar{B}=T\bar{N}$ be the opposite Borel subgroup, and $\xi:N(F)\rightarrow \BC^{\times}$ be a generic character. Then there exists a unique element $\Xi\in \bar{\Fn}(F)$ such that
$$\xi(\exp(X))=\psi(<X,\Xi>),\;X\in \Fn(F).$$
Without loss of generality, we assume that $G(F)$ has finite center (otherwise, we just need to replace $N(F)$ by $N(F)Z_{G}^{\circ}(F)$ where $Z_{G}^{\circ}(F)$ is the neutral component of $Z_G(F)$). For any irreducible smooth representation $\pi$ of $G(F)$, define the multiplicity
$$m(\pi,\xi)=\dim(\Hom_{N(F)}(\pi,\xi)).$$

Let $\CO\in Nil_{reg}(\Fg(F))$ be the nilpotent orbit containing $\Xi$. By the work of Rodier in \cite{Rod81}, we have the multiplicity formula
$$m(\pi,\xi)=c_{\theta_{\pi},\CO}(1).$$
The goal of this subsection is to show that
$$m_{geom}(\pi,\xi)=c_{\theta_{\pi},\CO}(1).$$

First, it is easy to see that the set $\CT(G,N)$ only contains the trivial torus. Combining with the fact that the Whittaker model is the Whittaker induction of the model $(T,1)$, we have
$$m_{geom}(\pi,\xi)=\frac{1}{|\CN(G,N,\xi)|} \sum_{\CO'\in \CN(G,N,\xi)} c_{\theta_{\pi},\CO'}(1).$$
Hence it is enough to show that
$$\CN(G,N,\xi)=\{\CO\}.$$
By the definition of the set $\CN(G,N,\xi)$, we have $\CO\in \CN(G,N,\xi)$. Let $\CO'\in Nil_{reg}(\Fg(F))$ with $\CO'\neq \CO$. It is enough to show that $\CO'\notin \CN(G,N,\xi)$. This will follow from the following lemma and Lemma \ref{Shalika germ}.

\begin{lem}
There exists a regular semisimple element $X\in \Fg_{reg}(F)$ such that
$$\Gamma_{\CO}(X)=1,\;\Gamma_{\CO'}(X)=0.$$
Here $\Gamma_{\CO}(\cdot)$ (resp. $\Gamma_{\CO'}(\cdot)$) is the Shalika germ defined in Section \ref{section shalika germ}.
\end{lem}

\begin{proof}
By the result of Shelstad in \cite{S}, the regular Shalika germ is equal to either 0 or 1. Hence if the statement of the lemma is false, we have $\Gamma_{\CO}(X)=\Gamma_{\CO'}(X)$ for all regular semisimple elements in $\Fg(F)$. By the result of Vign\'eras in \cite{V}, there exists $f\in C_{c}^{\infty}(\Fg(F))$ supported on regular elements (including regular nilpotent elements) such that $J_{\CO}(f)=1$, $J_{\CO'}(f)=-1$ and $J_{\CO_0}(f)=0$ for all other nilpotent orbits (not necessary regular). By replacing $f$ by $f\cdot 1_{\omega}$ where $\omega$ is a small $G$-invariant neighborhood of $0$ in $\Fg(F)$, we may assume that for all $X\in Supp(f)\cap \Fg_{reg}(F)$, we have
$$J_G(X,f)=\sum_{\CO_0\in Nil(\Fg(F))} \Gamma_{\CO_0}(X)J_{\CO_0}(f).$$
This implies that
$$J_G(X,f)=\sum_{\CO_0\in Nil(\Fg(F))} \Gamma_{\CO_0}(X)J_{\CO_0}(f)=\Gamma_{\CO}(X)-\Gamma_{\CO'}(X)=0$$
for all $X\in Supp(f)\cap \Fg_{reg}(F)$. Hence $J_G(X,f)=0$ for all $X\in \Fg_{reg}(F)$. By Theorem 3.1 of \cite{HC}, we know that $J_{\CO}(f)=J_{\CO'}(f)=0$. This is a contradiction.
\end{proof}

\begin{rmk}
\emph{In this case, $\CT(G,N)=\{1\}$ which implies that all regular semisimple conjugacy classes of $\Fg(F)$ are null with respect to $N$.}
\end{rmk}

\subsection{The Gan-Gross-Prasad model}\label{section GGP}
We only consider the orthogonal group case, the unitary group case is similar. We first recall the definition of the model from Section 7 of \cite{W10}. Let $V$ be a vector space of dimension $d$, and $q$ be a nondegenerate symmetric bilinear form on $V$. Let $r\in \BN$ with $2r+1\leq d$. Suppose we have an orthogonal decomposition $V=W\oplus D\oplus Z$ where $D$ is a one-dimensional anisotropic subspace and $Z$ is a hyperbolic subspace of dimension $2r$. We fix a basis $v_0$ of $D$ and a basis $(v_i)_{i=\pm1,\cdots, \pm r}$ of $Z$ with $q(v_i,v_j)=\delta_{i,-j}$. Let $A$ be the maximal split torus of $\SO(Z)$ that preserves the subspace $Fv_i$. Let $G=\SO(V)$, $P=MN$ be the parabolic subgroup of $G$ preserves the filtration
$$Fv_r\subset Fv_r\oplus Fv_{r-1}\subset \cdots \subset Fv_r\oplus \cdots \oplus Fv_1$$
with $A\subset M$. In particular, $M=AG_0$ with $G_0=\SO(V_0)$ and $V_0=W\oplus D$. Let $\xi:N(F)\rightarrow \BC^{\times}$ be the generic character defined in Section 7.2 of \cite{W10}. Its stabilizer in $M(F)$ is $H_{0}^{+}(F)=\mathrm{O}(W)$. Let $H_0=\SO(W)$ be the neutral component of $H_{0}^{+}$ and $H=H_0\ltimes N$. The model $(G\times H_0,H,\xi)$ is the so called Gan-Gross-Prasad model for orthogonal groups (the embedding $H\rightarrow G\times H_0$ comes from the diagonal embedding $H_0\rightarrow G_0\times H_0$ and the embedding $N\rightarrow G$) defined by Gross and Prasad in \cite{GP}. It is the Whittaker induction of the model $(G_0\times H_0, H_0)$ (which is also a Gan-Gross-Prasad model). Let $\pi$ (resp. $\sigma$) be an irreducible smooth representation of $G(F)$ (resp. $H_0(F)$). Define the multiplicity
$$m(\pi\otimes \sigma,\xi)=\dim(\Hom_{H(F)}(\pi\otimes\sigma,\xi)).$$
The multiplicity formula for this model was proved by Waldspurger in \cite{W10} and \cite{W12}. The goal of this subsection is to show that the geometric multiplicity $m_{geom}(\pi\otimes \sigma,\xi)$ defined in Section \ref{Section multiplicity formula} matches Waldspurger's definition in Section 13.1 of \cite{W10}. We use $m_{geom}'(\pi\otimes \sigma,\xi)$ to denote the geometric multiplicity defined by Waldspurger.

\begin{rmk}\label{GGP constant}
\emph{$(G_0\times H_0,H_0)$ is a minimal wavefront spherical variety. Moreover, it is easy to see that there is only one open Borel orbit in $G_0(F)\times H_0(F)/H_0(F)$ and it has trivial stabilizer. In particular, we have $d(G_0\times H_0,H_0,F)=c(G_0\times H_0,H_0,F)=1$.}
\end{rmk}

\begin{prop}
The set $\CT(G\times H_0,H)$ consists of tori $T(F)$ of $H_0(F)$ (up to conjugation) such that there exists an orthogonal decomposition $W=W'\oplus W''$ of $W$ satisfies the followings conditions.
\begin{enumerate}
\item $\dim(W')$ is an even number.
\item $T(F)$ is a maximal elliptic torus of $H_0'(F)=\SO(W')(F)$.
\item If $d$ is odd, the anisotropic rank of $V''=W''\oplus D\oplus Z$ is equal to 1. If $d$ is even, the anisotropic rank of $W''$ is equal to 1. This is equivalent to say that $\SO(V'')(F)$ and $\SO(W'')(F)$ are quasi-split.
\end{enumerate}
In particular, $\CT(G\times H_0,H)=\CT(G,H)^\circ$.
\end{prop}

\begin{rmk}
\emph{The proposition implies that the set $\CT(G\times H_0,H)$ is equal to the set $\ul{\CT}$ defined in Section 7.3 of \cite{W10}.}
\end{rmk}

\begin{proof}
It is easy to see that if a torus satisfies (1)-(3), it belongs to the set $\CT(G,H)$. So we only need to prove the other direction. For given $T(F)\in \CT(G,H)$, we need to show that $T(F)$ satisfies (1)-(3). Let $W''$ be the intersection of the kernel of $t-1$ for $t\in T(F)$. Then for almost all $t\in T_H(F)$, $W''$ is the kernel of $t-1$. In particular, $q|_{W''}$ is nondegenerate and $\dim(W)-\dim(W'')$ is an even number. Let $W'$ be the orthogonal complement of $W''$ in $W$ (i.e. $W=W'\oplus W''$), and $V''=W''\oplus D\oplus Z$. Then $T(F)$ is an abelian subgroup of $\SO(W')(F)$, $G_T=\SO(W')_T\times \SO(V'')$, $H_{0,T}=\SO(W')_{T}\times \SO(W'')$ and $H_T=\SO(W')_T\times (\SO(W'')\ltimes N'')$ where $N''=N\cap \SO(V'')$ is the unipotent radical of the parabolic subgroup $P''=P\cap \SO(V'')$ of $\SO(V'')$. In particular, $(\SO(V'')\times \SO(W''),\SO(W'')\ltimes N'')$ is the Gan-Gross-Prasad model associated to the decomposition $V''=W''\oplus D\oplus Z$. We will show that the decomposition $W=W'\oplus W''$ satisfies condition (1)-(3).

(1) follows from the fact that $\dim(W)-\dim(W'')$ is an even number. Since $G_T(F)$ and $H_{0,T}(F)$ are quasi-split, so are $\SO(V'')(F)$ and $\SO(W'')(F)$. This proves (3). It remains to prove (2). The following two statements follow from the definition of minimal spherical variety.
\begin{itemize}
\item If $(G_1,H_1)$ and $(G_2,H_2)$ are two spherical pairs, then $(G_1\times G_2,H_1\times H_2)$ is minimal if and only if $(G_1,H_1)$ and $(G_2,H_2)$ are minimal.
\item For any connected reductive group $H_1$, the spherical pair $(H_1\times H_1,H_1)$ is minimal if and only if $H_1$ is abelian (i.e. it is a torus).
\end{itemize}
Since $T(F)\in \CT(G,H)$, $(G_T\times H_{0,T},H_T)$ is minimal. By the statements above, we know that $\SO(W')_T$ is abelian which implies that $\SO(W')_T$ is a maximal torus of $\SO(W')$. By Definition \ref{defn support}(3), we know that $T(F)$ is the intersection of $H(F)$ with the center of $Z_G(T)(F)\times Z_{H_0}(T)(F)$, which implies that $T(F)=\SO(W')_T(F)$ (i.e. $T(F)=T^\circ(F)$ is a maximal torus of $\SO(W')(F)$). Finally, by Definition \ref{defn support}, we know that $T(F)$ is compact which implies that it is a maximal elliptic torus of $\SO(W')(F)$. This proves (2) and finishes the proof of the proposition.
\end{proof}

Given $T(F)\in \CT(G\times H_0,H)$ and let $W=W'\oplus W''$ be the decomposition associated to $T$. Then the model $(G_T\times H_{0,T},H)$ is the product of the abelian model $(\SO(W')_T,\SO(W')_T)=(T,T)$ and the Gan-Gross-Prasad model associated to the decomposition $V''=W''\oplus D\oplus Z$. By Remark \ref{GGP constant}, we know that the constants $d(G_{0,T}\times H_{0,T},H_{0,T},F)=c(G_{0,T}\times H_{0,T},H_{0,T},F)$ associated to the Gan-Gross-Prasad model are equal to 1. Moreover, since $Z_{H_0}(T)=H_{0,T}$, the constant $|Z_{H_0}(T)(F):H_{0,T}(F)|$ in the definition of geometric multiplicity is also equal to 1. Hence in order to prove $m_{geom}(\pi\otimes \sigma,\xi)=m_{geom}'(\pi\otimes \sigma,\xi)$, it remains to show that our choice of nilpotent orbits in Section \ref{Section nilpotent orbit} matches Waldspurger's choice in Section 7.3 of \cite{W10}.

\begin{prop}
Assume that $G(F)$ and $H_0(F)$ are quasi-split. Let $\CO_G$ (resp. $\CO_H$) be the regular nilpotent orbit of $\Fg(F)$ (resp. $\Fh_0(F)$) defined in Section 7.3 of \cite{W10}. Then we have
$$\CN(G\times H_0,H,\xi)=\{\CO_G\times \CO_H\}.$$
\end{prop}

\begin{proof}
Let $\Xi+\Fh_{0}^{\perp}(F)+\Fn(F)\subset \Fg(F)\oplus \Fh_0(F)$ be the space associated to the model $(G\times H_0,H,\xi)$ as in Section \ref{Section slice nonred}. By Lemma \ref{Shalika germ} together with Section 11.4-11.6 of \cite{W10}, we know that $\CO\notin \CN(G\times H_0,H,\xi)$ for any $\CO\in Nil_{reg}(\Fg(F)\times \Fh_0(F))$ with $\CO\neq \CO_G\times \CO_H$. In fact, for any $\CO\in Nil_{reg}(\Fg(F)\times \Fh_0(F))$ with $\CO\neq \CO_G\times \CO_H$, in Section 11.4-11.6 of \cite{W10}, Waldspurger has constructed an open subset $\Ft_G(F)$ (resp. $\Ft_H(F)$) of the regular semisimple conjugacy classes of $\Fg(F)$ (resp. $\Fh_0(F)$) such that for all $X_G\times X_H\in \Ft_G(F)\times \Ft_H(F)$, the followings hold.
\begin{itemize}
\item $\Gamma_{\CO}(X_G\times X_H)=1$ and the conjugacy class $X_G\times X_H$ has no intersection with $\Xi+\Fh_{0}^{\perp}(F)+\Fn(F)$.
\item $X_G\times X_H$ is null with respect to $H$.
\end{itemize}
Combining with Lemma \ref{Shalika germ}, we know that $\CO\notin \CN(G\times H_0,H,\xi)$.

Now it remains to show that $\CO_G\times \CO_H\in \CN(G\times H_0,H,\xi)$. The idea is to use the Lie algebra version of the local trace formula proved in \cite{W10}. Let $f_G$ (resp. $f_H$) be a smooth compactly supported strongly cuspidal function on $\Fg(F)$ (resp. $\Fh_0(F)$). Let $\theta_{f_G}$ (resp. $\theta_{f_H}$) be the quasi-character on $\Fg(F)$ (resp. $\Fh_0(F)$) associated to $f_G$ (resp. $f_H$), and $\hat{\theta}_{f_G}$ (resp. $\hat{\theta}_{f_H}$) be its Fourier transform. By the local trace formula proved in Section 11 of \cite{W10}, we have
\begin{equation}\label{1}
I(\theta_{f_H},\theta_{f_G})=\sum_{T\in \CT} |W(G,T)|^{-1} \int_{\Ft(F)^H} D^{G\times H_0}(t)^{1/2} \hat{\theta}_{f_G}\times \hat{\theta}_{f_H}(t) dt
\end{equation}
where $I(\theta_{f_H},\theta_{f_G})$ is the Lie algebra analogue of the geometric multiplicity defined in Section 7.9 of \cite{W10}, $\CT$ is the set of maximal tori of $G(F)\times H_0(F)$, and $W(G,T)=N_G(T)(F)/Z_G(T)(F)$ is the Weyl group. For $T\in \CT$, $\Ft^{H}(F)$ is the set of elements in $\Ft_{reg}(F)$ that is conjugated to an element in $\Xi+\Fh_{0}^{\perp}(F)+\Fn(F)$ (which is an open subset of $\Ft_{reg}(F)$).

If $\CO_G\times \CO_H\notin \CN(G\times H_0,H,\xi)$, by Lemma \ref{Shalika germ} and the definition of $\CN(G\times H_0,H,\xi)$, there exists $T_0\in \CT$ and a small open compact subset $\omega$ of $\Ft_{0,reg}(F)$ satisfies the following two conditions
\begin{itemize}
\item For all $X\in \omega$, $X$ is null with respect to $H$ and $X$ is associated to $\CO_G\times \CO_H$.
\item The set $\omega'=\{X\in \omega|\; X\notin \Ft_0(F)^H\}$ has nonzero measure.
\end{itemize}
Now choose $f_G$ and $f_H$ such that $\hat{\theta}_{f_G}\times \hat{\theta}_{f_H}$ is the characteristic function on $\omega^{G\times H_0}$. Then the right hand side of \eqref{1} is equal to
\begin{equation}\label{2}
\int_{\omega\cap \Ft_0(F)^H} D^{G\times H_0}(t)^{1/2}dt.
\end{equation}
Since every element in $\omega$ is null with respect to $H$ and is associated to $\CO_G\times \CO_H$, by Proposition 4.1.1 and 4.7.1 of \cite{B15}, we have
$$I(\theta_{f_H},\theta_{f_G})=c_{\theta_{f_G}\times \theta_{f_H},\CO_G\times \CO_H}(0)=\int_{\omega} D^{G\times H_0}(t)^{1/2}\Gamma_{\CO_G\times \CO_H}(t)dt$$
$$=\int_{\omega} D^{G\times H_0}(t)^{1/2}dt=\int_{(\omega\cap \Ft_0(F)^H)\cup \omega'} D^{G\times H_0}(t)^{1/2}dt.$$
This is a contradiction to \eqref{1} and \eqref{2} since $\omega'$ has nonzero measure. Hence $\CO_G\times \CO_H\in \CN(G\times H_0,H,\xi)$. This finishes the proof of the proposition.
\end{proof}

\subsection{The Ginzburg-Rallis model}\label{section GR}
In this subsection, we consider the Ginzburg-Rallis model case. We will show that the geometric multiplicity defined in Section \ref{Section multiplicity formula} matches the one in the multiplicity formula proved in \cite{Wan15}, \cite{Wan16} (the general linear group case) and \cite{WZ} (the unitary and unitary similitude group case). For simplicity, we only consider the quasi-split unitary group and unitary similitude group cases, the non quasi-split case and the general linear group case follows from a similar and easier argument.

Set $w_2=\begin{pmatrix}0&1\\1&0\end{pmatrix}$, and $w_n=\begin{pmatrix}0&1\\w_n&0\end{pmatrix}$ for $n>2$. Let $E/F$ be a quadratic extension. We define the unitary group and unitary similitude group to be
$$\mathrm{U}_n(F)=\{g\in \GL_n(E)|\; \bar{g}^t w_ng=w_n\},\;\mathrm{GU}_n(F)=\{g\in \GL_n(E)|\; \bar{g}^t w_ng=\lambda w_n,\;\lambda\in F^{\times}\}.$$
We use $\lambda:\mathrm{GU}_n(F)\rightarrow F^{\times}$ to denote the similitude character.

\subsubsection{The unitary similitude group case}
Let $G(F)=\mathrm{GU}_6(F)$, $H(F)=H_0(F)\ltimes N(F)$ with
$$H_0(F)=\{\begin{pmatrix} h&0&0\\0&h&0\\0&0&\lambda(h) w_2(\bar{g}^t)^{-1}w_2 \end{pmatrix}|\; h\in \mathrm{GU}_2(F)\},$$
$$N(F)=\{\begin{pmatrix} I_2&X&Y\\0&I_2&-w_2\bar{X}^t w_2\\ 0&0&I_2\end{pmatrix} |\; X,Y\in Mat_{2\times 2}(E), w_2Xw_2\bar{X}^t+w_2Yw_2+\bar{Y}^t=0\}.$$
Let $\chi$ be a character of $\mathrm{GU}_2(F)$. Define the character $\omega\otimes \xi$ on $H(F)$ to be
$$\omega\otimes \xi(\begin{pmatrix} h&0&0\\0&h&0\\0&0&\lambda(h) w_2(\bar{g}^t)^{-1}w_2 \end{pmatrix} \begin{pmatrix} I_2&X&Y\\0&I_2&-w_2\bar{X}^t w_2\\ 0&0&I_2\end{pmatrix})=\chi(h)\psi(\tr_{E/F}(\tr(X))).$$
Let $\pi$ be an irreducible smooth representation of $G(F)$. Define the multiplicity
$$m(\pi,\omega\otimes \xi)=\dim(\Hom_{H(F)} (\pi,\omega\otimes \xi)).$$
The model $(G,H)$ is the unitary similitude analogue of the Ginzburg-Rallis model defined in \cite{GR}, and it is the Whittaker induction of the model $(G_0,H_0,\xi)=(\mathrm{GU}_2(F)\times \GL_2(E),\mathrm{GU}_2(F),\xi)$. Also it is easy to see that both $(G,H)$ and $(G_0,H_0)$ are minimal.

In \cite{WZ}, we proved the multiplicity formula
$$m(\pi,\omega\otimes \xi)=c_{\theta_{\pi},\CO_{reg}}(1)+\sum_{T\in \CT_{ell}(H_0)} |W(H_0,T)|^{-1} \int_{T(F)/A_{H_0}(F)} \omega(t)^{-1} D^H(t) c_{\theta_{\pi},\CO_t}(t)dt$$
where $\CO_{reg}$ is the unique regular nilpotent orbit of $\Fg(F)$, $\CT_{ell}(H_0)$ is the set of all maximal elliptic tori of $H_0(F)$, and for $T\in \CT_{ell}(H_0)$, $t\in T(F)_{reg}$, $\CO_t$ is the unique regular nilpotent orbit in $\Fg_t(F)$. The goal of this subsection is to show that
\begin{equation}\label{similitude unitary GR}
m_{geom}(\pi,\omega\otimes \xi)=c_{\theta_{\pi},\CO_{reg}}(1)+\sum_{T\in \CT_{ell}(H_0)} |W(H_0,T)|^{-1} \int_{T(F)/A_{H_0}(F)} \omega(t)^{-1} D^H(t) c_{\theta_{\pi},\CO_t}(t)dt.
\end{equation}

First, it is easy to see from the definition that $\CT(G,H)=\CT(G,H)^\circ=\CT_{ell}(H_0)\cup \{1\}$. For $T\in \CT_{ell}(H_0)$, $G_T=Z_G(T),\; H_{0,T}=Z_{H_0}(T)$, and the model $(G_T,H_T,\xi)$ is just the Whittaker model of $G_T$. By the result in Section \ref{section Whittaker} for the Whittaker model, we only need to consider the geometric multiplicity at $T=\{1\}$ and it is enough to prove the following lemma.

\begin{lem}\label{unitary GR lemma}
\begin{enumerate}
\item $d(G_0,H_0,F)=c(G_0,H_0,1)=1$.
\item $\CN(G,H,\xi)=\{\CO_{reg}\}$.
\end{enumerate}
\end{lem}

\begin{proof}
It is easy to see that there is only one open Borel orbit in $G_0(F)/H_0(F)$ and the stabilizer of this orbit is the center of $H_0(F)$ which is connected. This implies that $d'(G_0,H_0,F)=c(G_0,H_0,F)=1$. On the other hand, the model $(G_0(\bar{F}),H_0(\bar{F}))$ is essentially the trilinear $\GL_2$ model $(\GL_2\times \GL_2\times \GL_2,\GL_{2}^{diag})$ which is wavefront. Hence $d(G_0,H_0,F)=d'(G_0,H_0,F)=1$. This proves (1). For (2), the argument is very similar to the Gan-Gross-Prasad model case. We just need to use the local trace formula for the model $(G,H)$ proved in \cite{WZ}. We will skip the details here. This finishes the proof of the lemma and hence the proof of \eqref{similitude unitary GR}.
\end{proof}

\subsubsection{The unitary group case}
Let $G(F)=\mathrm{U}_6(F)$, $H(F)=H_0(F)\ltimes N(F)$ with
$$H_0(F)=\{\begin{pmatrix} h&0&0\\0&h&0\\0&0& w_2(\bar{g}^t)^{-1}w_2 \end{pmatrix}|\; h\in \mathrm{U}_2(F)\},$$
$$N(F)=\{\begin{pmatrix} I_2&X&Y\\0&I_2&-w_2\bar{X}^t w_2\\ 0&0&I_2\end{pmatrix} |\; X,Y\in Mat_{2\times 2}(E), w_2Xw_2\bar{X}^t+w_2Yw_2+\bar{Y}^t=0\}.$$
Let $\chi$ be a character of $\mathrm{U}_2(F)$. Define the character $\omega\otimes \xi$ on $H(F)$ to be
$$\omega\otimes \xi(\begin{pmatrix} h&0&0\\0&h&0\\0&0& w_2(\bar{g}^t)^{-1}w_2 \end{pmatrix} \begin{pmatrix} I_2&X&Y\\0&I_2&-w_2\bar{X}^t w_2\\ 0&0&I_2\end{pmatrix})=\chi(h)\psi(\tr_{E/F}(\tr(X))).$$
Let $\pi$ be an irreducible smooth representation of $G(F)$. Define the multiplicity
$$m(\pi,\omega\otimes \xi)=\dim(\Hom_{H(F)} (\pi,\omega\otimes \xi)).$$
The model $(G,H)$ is the unitary analogue of the Ginzburg-Rallis model and it is the Whittaker induction of the model $(G_0,H_0,\xi)=(\mathrm{U}_2(F)\times \GL_2(E),\mathrm{U}_2(F),\xi)$. Also it is easy to see that both $(G,H)$ and $(G_0,H_0)$ are minimal.

In \cite{WZ}, we proved the multiplicity formula
$$m(\pi,\omega\otimes \xi)=c_{\theta_{\pi},\CO_{reg,1}}(1)+c_{\theta_{\pi},\CO_{reg,2}}(1)+\sum_{T\in \CT_{ell}(H_0)} |W(H_0,T)|^{-1} \int_{T(F)} \omega(t)^{-1} D^H(t) c_{\theta_{\pi},\CO_t}(t)dt$$
where $\CO_{reg,1},\CO_{reg,2}$ are the regular nilpotent orbits of $\Fg(F)$, $\CT_{ell}(H_0)$ is the set of all maximal elliptic tori of $H_0(F)$, and for $T\in \CT_{ell}(H_0)$, $t\in T(F)_{reg}$, $\CO_t$ is the unique regular nilpotent orbit in $\Fg_t(F)$. The goal of this subsection is to show that
\begin{equation}\label{GR unitary}
m_{geom}(\pi,\omega\otimes \xi)=c_{\theta_{\pi},\CO_{reg,1}}(1)+c_{\theta_{\pi},\CO_{reg,2}}(1)+\sum_{T\in \CT_{ell}(H_0)} |W(H_0,T)|^{-1}\int_{T(F)} \omega(t)^{-1} D^H(t) c_{\theta_{\pi},\CO_t}(t)dt.
\end{equation}
By the same argument as in the unitary similitude group case,  we only need to prove the following lemma.

\begin{lem}
\begin{enumerate}
\item $d(G_0,H_0,F)=2,\;c(G_0,H_0,F)=1$.
\item $\CN(G,H,\xi)=\{\CO_{reg,1},\;\CO_{reg,2}\}$.
\end{enumerate}
\end{lem}

\begin{proof}
It is easy to see that there are two open Borel orbits of $G_0(F)/H_0(F)$ (corresponds to $F^{\times}/Im(N_{E/F})$ where $N_{E/F}:E^{\times} \rightarrow F^{\times}$ is the norm map) and the stabilizer of each orbit is the center of $H_0(F)$ which is connected. This implies that $d'(G_0,H_0,F)=2$ and $c(G_0,H_0,F)=1$. On the other hand, the model $(G_0(\bar{F}),H_0(\bar{F}))$ is the trilinear $\GL_2$ model which is wavefront. Hence $d(G_0,H_0,F)=d'(G_0,H_0,F)=2$. This proves (1).
\\
\\
For (2), we can not use the same argument as in the previous cases. The reason is that in \cite{WZ}, we were not able to prove the local trace formula for this model (this is largely due to the fact that the number $d(G_0,H_0,F)$ is not equal to 1, see Remark \ref{rmk stable}). Instead, we are going to use the result for the unitary similitude group case to prove (2).

Let $\Xi+\Fh_{0}^{\perp}(F)+\Fn(F)$ be the space associated to the model $(G\times H_0,H,\xi)$ as in Section \ref{Section slice nonred}. Let $\Fg'(F)$ be the Lie algebra of $\mathrm{GU}_6(F)$, $\CO_{reg}$ be the unique nilpotent orbit of $\Fg'(F)$, and $(G',H',\xi)$ be the model in the unitary similitude group case. Then $\CO_{reg}=\CO_{reg,1}\cup \CO_{reg,2}$ and $\Fg'(F)=\Fg(F)\oplus \Fz(F)$ where $\Fz(F)=\{aI_6|\; a\in F\}$ belongs to the center of $\Fg'(F)$. Moreover, $\Xi+\Fh_{0}^{\perp}(F)+\Fn(F)+\Fz(F)$ is the space associated the model $(G',H',\xi)$.

Since $\CO=\CO_{reg,1}\cup \CO_{reg,2}$, if a regular semisimple element $X\in \Fg(F)$ is associate to $\CO_{reg,1}$ (resp. $\CO_{reg,2}$), then it is associated to $\CO$ (as an element in $\Fg'(F)$). Moreover, $X$ is null with respect to $H$ if and only if it is null with respect to $H'$. Hence by Lemma \ref{unitary GR lemma}, we know that for almost all regular semisimple $G(F)'$-conjugacy classes in $\Fg(F)$, if the conjugacy class is null with respect to $H$ and if it is associated to $\CO_{reg,1}$ (resp. $\CO_{reg,2}$), then the conjugacy class has nonempty intersection with $\Xi+\Fh_{0}^{\perp}(F)+\Fn(F)$. As a result, in order to prove the lemma, it is enough to prove the following statement.

\begin{itemize}
\item[(3)] For all regular semisimple elements $X_1,X_2\in \Fg_{reg}(F)$, if $X_1$ and $X_2$ are null with respect to $H$, then $X_1$ and $X_2$ are $G'(F)$-conjugated to each other if and only if they are $G(F)$-conjugated to each other.
\end{itemize}
In fact, since $X_1$ is null with respect to $H$, it is not elliptic regular semisimple. Let $T(F)=G_{X_1}'(F)$, and $A_T(F)$ be the maximal split subtorus of $T(F)$. Then $L(F)=Z_{G'}(A_T)(F)$ is a proper Levi subgroup of $G'(F)$. We have $X_1\in \Fl(F)$. In particular, $X_1$ commutes with $Z_L(F)$. Then (3) follows from the fact that every element $g\in G'(F)$ can be written as $g=g_1 z$ with $g_1\in G(F)$ and $z\in Z_L(F)$. This finishes the proof of the lemma and hence the proof of \eqref{GR unitary}.
\end{proof}

\subsection{The Galois model}\label{section Galois}
Let $E/F$ be a quadratic extension, $H$ be a connected reductive group defined over $F$, and $G=Res_{E/F}H$. Let $\chi$ be a character of $H(F)$. For any irreducible smooth representation $\pi$ of $G(F)$, define the multiplicity
$$m(\pi,\chi)=\dim(\Hom_{H(F)} (\pi,\chi)).$$
In \cite{B18}, Beuzart-Plessis proved the multiplicity formula for this model
$$m(\pi,\chi)=\sum_{T\in \CT_{ell}(H)} |W(H,T)|^{-1} \int_{T(F)/A_H(F)} \chi(t)^{-1} D^H(t)\theta_{\pi}(t) dt $$
where $\CT_{ell}(H)$ is the set of all maximal elliptic tori of $H(F)$. We want to show that
\begin{equation}\label{Galois model}
m_{geom}(\pi,\chi)=\sum_{T\in \CT_{ell}(H)} |W(H,T)|^{-1} \int_{T(F)/A_H(F)} \chi(t)^{-1} D^H(t)\theta_{\pi}(t)dt.
\end{equation}

For $T\in \CT_{ell}(H)$, $H_T(F)=Z_H(T)(F)=T(F)$ and the model $(G_T(F),H_T(F))$ is equal to the abelian model $(T(E),T(F))$. This implies that $|Z_H(T)(F):H_T(F)|=d(G_T,H_T,F)=c(G_T,H_T,F)=1$ and $\CN(G_T,H_T)=\{0\}$. Hence in order to prove \eqref{Galois model}, it is enough to show that the set $\CT(G,H)$ is equal to $\CT_{ell}(H)$. It is easy to see from the definition that $\CT_{ell}(H)\subset \CT(G,H)$. For the other direction, let $T(F)\in \CT(G,H)$. Then $(G_T,H_T)=(Res_{E/F}H_T, H_T)$. In particular, it is minimal if and only if $H_T$ is abelian (i.e. it is a maximal torus of $H$). By Definition \ref{defn support}(3), we know that $T(F)=T^\circ(F)=H_T(F)$ is a maximal torus of $H(F)$. By Definition \ref{defn support}(4), we know that $T(F)/A_H(F)$ is compact. This implies that $T\in \CT_{ell}(H)$ and proves \eqref{Galois model}.

\subsection{The Shalika model}\label{section Shalika}
Let $G=\GL_{2n}$ and $H=H_0\ltimes N$ with
$$H_0=\{\begin{pmatrix}h&0\\0&h\end{pmatrix}|\; h\in \GL_n\},\; N=\{\begin{pmatrix}I_n&X\\0&I_n\end{pmatrix}|\; X\in Mat_{n\times n}\}.$$
Given a multiplicative character $\chi:F^{\times}\rightarrow \BC^{\times}$, we can define a character $\omega\otimes \xi$ of $H(F)$ to be
$$\omega\otimes\xi(\begin{pmatrix}h&0\\0&h\end{pmatrix}  \begin{pmatrix}I_n&X\\0&I_n\end{pmatrix}) :=\psi(\tr(X))\chi(\det(h)).$$
For any irreducible smooth representation $\pi$ of $G(F)$, define the multiplicity
$$m(\pi,\omega\otimes\xi)=\dim(\Hom_{H(F)} (\pi,\omega\otimes\xi)).$$
The pair $(G,H)$ is called the Shalika model, it is the Whittaker induction of the model $(H_0\times H_0,H_0,\xi)=(\GL_n\times \GL_n,\GL_n,\xi)$. In a joint work with Beuzart-Plessis \cite{BW}, we have proved the multiplicity formula
$$m(\pi,\omega\otimes\xi)=\sum_{T\in \CT_{ell}(H_0)} |W(H_0,T)|^{-1} \int_{T(F)/Z_G(F)} \omega(t)^{-1}D^{H}(t)c_{\theta_{\pi},\CO_t}(t)dt$$
where $\CT_{ell}(H_0)$ is the set of all maximal elliptic tori of $H_0(F)$, and for $T\in \CT_{ell}(H_0)$, $t\in T(F)_{reg}$, $\CO_t$ is the unique regular nilpotent orbit in $\Fg_t(F)$. We want to show that
\begin{equation}\label{Shalika model}
m_{geom}(\pi,\omega\otimes\xi)=\sum_{T\in \CT_{ell}(H_0)} |W(H_0,T)|^{-1} \int_{T(F)/Z_G(F)} \omega(t)^{-1}D^{H}(t)c_{\theta_{\pi},\CO_t}(t)dt.
\end{equation}

For $T\in \CT_{ell}(H_0)$, let $K/F$ be the degree $n$ extension such that $T(F)\simeq K^{\times}$. Then the model $(G_T,H_T,\xi)$ is the just the Whittaker model for $\GL_2(K)$. By the result in Section \ref{section Whittaker} for the Whittaker model, we know that in order to prove \eqref{Shalika model}, it is enough to show that $\CT(G,H)=\CT_{ell}(H_0)$. It is clear that $\CT_{ell}(H_0)\subset \CT(G,H)$. For the other direction, let $T\in \CT(G,H)$. The model $(G_T,H_T)$ is the Whittaker induction of the model $(H_{0,T}\times H_{0,T},H_{0,T},\xi)$. Since it is minimal, we know that $H_{0,T}$ is abelian (i.e. it is a torus of $H_0$). By the same argument as in the Galois model case, we have $T\in \CT_{ell}(H_0)$. This proves \eqref{Shalika model}.

\section{The proof of Theorem \ref{main K-type}(1)}
\subsection{The geometric multiplicity}
Let $F=\BR,\;G=\GL_n$ and $H=\SO_n=\{g\in \GL_n|\; gg^{t}=I_n\}$. Then $H(\BR)$ is a maximal connected compact subgroup of $G(\BR)$. Let $\pi$ be a finite length smooth representation of $\GL_n(\BR)$ and $\omega$ be a finite dimensional representation of $\SO_n(\BR)$. The goal of this section is to prove Theorem \ref{main K-type}(1). In other words, we need to prove the multiplicity formula
$$m(\pi,\omega)=m_{geom}(\pi,\omega)$$
where $m(\pi,\omega)=\dim(\Hom_{\SO(F)}(\pi,\omega))$ and the geometric multiplicity $m_{geom}(\pi,\omega)$ was defined in Section \ref{section K-types}. In this subsection, we will give an explicit expression of $m_{geom}(\pi,\omega)$ (see Proposition \ref{geom multiplicity GL(n) SO(n)}).

\begin{defn}
$I(n)=\{(n_1,n_2,k)\in (\BZ_{\geq 0})^3|\; n_1+2n_2+2k=n\}.$ For $(n_1,n_2,k)\in I(n)$, if $n$ is even ($\iff$ $n_1$ is even), let $T_{n_1,n_2,k}$ be the abelian subgroup of $\SO_n(\BR)$ defined by
$$T_{n_1,n_2,k}(\BR)=\{diag(\pm I_{n_1},\pm I_{2n_2},t)|\; t\in (\BC^1)^k\}$$
where $\BC^1$ is the group of norm 1 element in $\BC$ and we identify it with $\SO_2(\BR)$ via the isomorphism $e^{2\pi i\theta}\mapsto \begin{pmatrix} \cos \theta & \sin \theta \\ -\sin \theta & \cos \theta \end{pmatrix}$. In particular, $t\in (\BC^1)^k$ becomes an element of $\SO_{2k}(\BR)\subset \GL_{2k}(\BR)$ and $diag(\pm I_{n_1},\pm I_{2n_2},t)$ are elements of $\SO_n(\BR)\subset \GL_n(\BR)$.

Similarly, if $n$ is odd ($\iff$ $n_1$ is odd), we define
$$T_{n_1,n_2,k}(\BR)=\{diag(I_{n_1},\pm I_{2n_2},t)|\; t\in (\BC^1)^k\}\subset \SO_n(\BR).$$
\end{defn}

\begin{lem}\label{support GL(n),SO(n) even}
Assume that $n$ is even. The set $\CT(G,H)$ (defined in Definition \ref{defn support}) is the union of $T_{n_1,n_2,k}(\BR)$ where $(n_1,n_2,k)\in I(n)$ with $n_1\geq 2n_2$.
\end{lem}

\begin{proof}
It is easy to see that $T_{n_1,n_2,k}(\BR)\in \CT(G,H)$. So it is enough to prove the other direction. Let $t$ be a semisimple element of $H(\BR)=\SO_n(\BR)$ such that $(G_t,H_t)$ is a minimal spherical pair. After conjugation, we may assume that $t=diag(I_{n_1},-I_{2n_2},t_0)$ where $t_0$ is a semisimple element in $\SO_{2k}(\BR)$ such that $t_0\pm I_{2k}\in \GL_{2k}(\BR)$ (i.e. $\pm 1$ are not the eigenvalues of $t_0$). Here $2k=n-n_1-2n_2$.

Since $\pm 1$ are not the eigenvalues of $t_0$, the centralizer of $t_0$ in $\GL_{2k}(\BR)$ is of the form (note that all the eigenvalues of $t$ belong to $\BC^1$)
$$\GL_{k_1}(\BC)\times \cdots \times \GL_{k_{m}}(\BC)$$
with $k=k_1+\cdots +k_m$. Then
$$G_t(\BR)=\GL_{n_1}(\BR)\times \GL_{2n_2}(\BR)\times \GL_{k_1}(\BC)\times \cdots \times \GL_{k_{m}}(\BC),$$
$$H_t(\BR)=\SO_{n_1}(\BR)\times \SO_{2n_2}(\BR)\times  \mathrm{U}_{k_1}(\BR) \times \cdots \times \mathrm{U}_{k_{m}(\BR)}.$$
Since $(G_t,H_t)$ is a minimal spherical pair, we know that $(Res_{\BC/\BR}\GL_{k_i},\mathrm{U}_{k_i})$ is a minimal spherical pair for $1\leq i\leq m$. This implies that $k_i=1$ for $1\leq i\leq m$. In other words, $t_0$ is a regular semisimple element of $\GL_{2k}(\BR)$.

Now we are ready to prove the lemma. Let $T(\BR)\in \CT(G,H)$. By conditions (1) and (4) of Definition \ref{defn support}, there exists $t\in T(\BR)$ such that $(G_T,H_T)=(G_t,H_t)$ is a minimal spherical pair. By the discussion above, up to conjugation, we may assume that $t=diag(I_{n_1},-I_{2n_2}, t_0)$ where $t_0\in \SO_{2k}(\BR)$ is a regular semisimple element of $\GL_{2k}(\BR)$ and $(n_1,n_2,k)\in I(n)$. Combining with condition (2) of Definition \ref{defn support}, we have
$$T(\BR)=Z_{G_t}(\BR)\cap H(\BR)=\{diag(\pm I_{n_1},\pm I_{2n_2},t')|\;t'\in T_0(\BR)\}$$
where $T_0(\BR)$ is the centralizer of $t_0$ in $\SO_{2k}(\BR)$ which is a maximal torus of $\SO_{2k}(\BR)$. Up to conjugation, we may assume that $n_1\geq 2n_2$. Then the lemma follows from the fact that every maximal torus of $\SO_{2k}(\BR)$ is conjugated to the torus $(\BC^1)^k$. This proves the lemma.
\end{proof}

\begin{lem}\label{support GL(n),SO(n) odd}
Assume that $n$ is odd. Then the set $\CT(G,H)$ is the union of $T_{n_1,n_2,k}(\BR)$ where $(n_1,n_2,k)\in I(n)$.
\end{lem}

\begin{proof}
The proof is similar to the previous lemma, we will skip it here.
\end{proof}

\begin{cor}\label{support GL(n),SO(n)}
The geometric multiplicity $m_{geom}(\pi,\omega)$ is supported on
$$\{diag(I_{n_1},-I_{2n_2},t)|\; t\in (\BC^1)^k\}\cup \{diag(-I_{n_1},I_{2n_2},t)|\; t\in (\BC^1)^k\},\; (n_1,n_2,k)\in I(n)\;\text{with}\; n_1\geq 2n_2$$
when $n$ is even; and it is supported on
$$\{diag(I_{n_1},-I_{2n_2},t)|\; t\in (\BC^1)^k\},\; (n_1,n_2,k)\in I(n)$$
when $n$ is odd.
\end{cor}

\begin{proof}
This is a direct consequence of the previous two lemmas.
\end{proof}

\begin{lem}\label{GL(n) SO(n) lemma 1}
\begin{enumerate}
\item $(G,H)$ is a minimal spherical pair.
\item $d(G,H,\BR)=1,\; c(G,H,\BR)=2^{n-1}$.
\item $\CN(G,H,1)=\{\CO\}$ where $\CO$ is the unique regular nilpotent orbit of $\Fg(\BR)=\Fg\Fl_n(\BR)$.
\end{enumerate}
\end{lem}

\begin{proof}
(1) is trivial. For (2), let $B(\BR)$ be the upper triangular Borel subgroup of $G(\BR)$. Since $(G,H)$ is a symmetric pair which is wavefront, we have $d(G,H,\BR)=d(G,H,\BR)'$. By the Iwasawa decomposition, we have $G(\BR)=B(\BR)H(\BR)$ and $B(\BR)\cap H(\BR)\simeq (\BZ/2\BZ)^{n-1}$. This implies that $d(G,H,\BR)=d(G,H,\BR)'=1$ and $c(G,H,\BR)=2^{n-1}$. (3) follows from the argument in the end of Section 5.
\end{proof}

Given $(n_1,n_2,k)\in I(n)$, and let $T=T_{n_1,n_2,k}$. Then the model $(G_T,H_T)$ is the product of the models $(\GL_{n_1}(\BR),\SO_{n_1}(\BR))$, $(\GL_{2n_2}(\BR),\SO_{2n_2}(\BR))$ and $((\BC^1)^k,(\BC^1)^k)$. The following lemma is easy to verify.

\begin{lem}\label{GL(n) SO(n) lemma 2}
\begin{enumerate}
\item The number $|Z_{H}(T)(\BR):H_T(\BR)|$ is equal to 1 if $n_1n_2=0$, and is equal to 2 if $n_1n_2\neq 0$.
\item If $n_1=n_2=0$ (this only happens when $n$ is even), then $|W(H,T)|=2^{k-1}k!=2^{n-k-n_1-2n_2-1}k!$. If $n_1=2n_2\neq 0$ (this only happens when $n$ is even and $n\geq 4$), then $|W(H,T)|=2\times 2^{k}k!=2^{n-k-n_1-2n_2+1}k!$. If $n_1\neq 2n_2$, then $|W(H,T)|=2^kk!=2^{n-k-n_1-2n_2}k!$.
\end{enumerate}
\end{lem}

Combining Corollary \ref{support GL(n),SO(n)}, Lemma \ref{GL(n) SO(n) lemma 1} and Lemma \ref{GL(n) SO(n) lemma 2}, we have
$$m_{geom}(\pi,\omega)=\sum_{(n_1,n_2,k)\in I(n),n_1>2n_2} \frac{1}{2^{n-k-1} k!} \int_{(\BC^1)^k} D^{\SO_n}(diag(I_{n_1},-I_{2n_2},t)) c_{\pi}(diag(I_{n_1},-I_{2n_2},t))$$
$$\theta_{\omega^{\vee}}(diag(I_{n_1},-I_{2n_2},t)) + D^{\SO_n}(diag(-I_{n_1},I_{2n_2},t)) c_{\pi}(diag(-I_{n_1},I_{2n_2},t)) \theta_{\omega^{\vee}}(diag(-I_{n_1},I_{2n_2},t))dt$$
$$+ \sum_{(n_1,n_2,k)\in I(n),n_1=2n_2\neq 0} \frac{1}{2^{n-k} k!} \int_{(\BC^1)^k}D^{\SO_n}(diag(I_{n_1},-I_{2n_2},t)) c_{\pi}(diag(I_{n_1},-I_{2n_2},t))$$
$$\theta_{\omega^{\vee}}(diag(I_{n_1},-I_{2n_2},t))+D^{\SO_n}(diag(-I_{n_1},I_{2n_2},t)) c_{\pi}(diag(-I_{n_1},I_{2n_2},t))\theta_{\omega^{\vee}}(diag(-I_{n_1},I_{2n_2},t))dt$$
$$+\frac{1}{2^{n-\frac{n}{2}-1} (\frac{n}{2}!)} \int_{(\BC^1)^{\frac{n}{2}}} D^{\SO_n}(t)c_{\pi}(t) \theta_{\omega^{\vee}}(t)dt$$
when $n$ is even, and
\begin{eqnarray*}
m_{geom}(\pi,\omega)&=&\sum_{(n_1,n_2,k)\in I(n)} \frac{1}{2^{n-k-1} k!} \int_{(\BC^1)^k} D^{\SO_n}(diag(I_{n_1},-I_{2n_2},t)) \\ &&c_{\pi}(diag(I_{n_1},-I_{2n_2},t)) \theta_{\omega^{\vee}}(diag(I_{n_1},-I_{2n_2},t)) dt
\end{eqnarray*}
when $n$ is odd where
\begin{itemize}
\item The Haar measure on $\BC^1=\SO_2(\BR)$ is the one that makes the total volume equal to 1
\item $c_{\pi}(diag(I_{n_1},-I_{2n_2},t))$ (resp. $c_{\pi}(diag(-I_{n_1},I_{2n_2},t))$) is the regular germ of $\theta_{\pi}$ of $\pi$ at $diag(I_{n_1},-I_{2n_2},t)$ (resp. $diag(-I_{n_1},I_{2n_2},t)$) defined in Section \ref{section germ parabolic}.
\item $\omega^{\vee}$ is the dual representation of $\omega$ and $\theta_{\omega^{\vee}}$ is the character of $\omega^{\vee}$.
\end{itemize}

When $n$ is even, we can replace the element $diag(-I_{n_1},I_{2n_2},t)$ in the expression of $m_{geom}(\pi,\omega)$ by $diag(I_{2n_2},-I_{n_1},t)$ because they are conjugated to each other in $\SO_n(\BR)$. Then we have
\begin{eqnarray*}
m_{geom}(\pi,\omega)&=&\sum_{(n_1,n_2,k)\in I(n)} \frac{1}{2^{n-k-1} k!} \int_{(\BC^1)^k} D^{\SO_n}(diag(I_{n_1},-I_{2n_2},t)) \\ &&c_{\pi}(diag(I_{n_1},-I_{2n_2},t)) \theta_{\omega^{\vee}}(diag(I_{n_1},-I_{2n_2},t)) dt.
\end{eqnarray*}
In other words, we get the same expression as in the odd case. To summarize, we have proved the following proposition.

\begin{prop}\label{geom multiplicity GL(n) SO(n)}
\begin{eqnarray*}
m_{geom}(\pi,\omega)&=&\sum_{(n_1,n_2,k)\in I(n)} \frac{1}{2^{n-k-1} k!} \int_{(\BC^1)^k} D^{\SO_n}(diag(I_{n_1},-I_{2n_2},t)) \\ &&c_{\pi}(diag(I_{n_1},-I_{2n_2},t)) \theta_{\omega^{\vee}}(diag(I_{n_1},-I_{2n_2},t)) dt.
\end{eqnarray*}
\end{prop}

\subsection{A reduction}
Given a finite length smooth representation $\pi$ of $\GL_n(\BR)$ and a finite dimensional representation $\omega$ of $\SO_n(\BR)$, we need to prove the multiplicity formula
\begin{equation}\label{3.1}
m(\pi,\omega)=m_{geom}(\pi,\omega)
\end{equation}
where $m_{geom}(\pi,\omega)$ was defined in Proposition \ref{geom multiplicity GL(n) SO(n)}.

In order to prove \eqref{3.1}, we need a multiplicity formula for the model $(\GL_n(\BR),\mathrm{O}_n(\BR))$. To be specific, let $\omega_+$ be a finite dimensional representation of $\mathrm{O}_n(\BR)=\{g\in \GL_n(\BR)|\; gg^t=I_n\}$, $\omega_{+}^{\vee}$ be the dual representation, and $\theta_{\omega_{+}^{\vee}}:\mathrm{O}_n(\BR)\rightarrow \BC$ be the character of $\omega_{+}^{\vee}$. We use $sgn:\; \mathrm{O}_n(\BR)\rightarrow \{\pm1\}$ to denote the sign character of $\mathrm{O}_n(\BR)$. Given a finite length smooth representation $\pi$ of $\GL_n(\BR)$, we define the multiplicity
$$m(\pi,\omega_+)=\dim(\Hom_{\mathrm{O}_n(\BR)}(\pi,\omega_+)),$$
and the geometric multiplicity
\begin{eqnarray}\label{3.3}
m_{geom}(\pi,\omega_+)&=&\sum_{(n_1,n_2,k)\in J(n)} \frac{1}{2^{n-k} k!} \int_{(\BC^1)^k} D^{\SO_n}(diag(I_{n_1},-I_{n_2},t)) \nonumber \\ &&c_{\pi}(diag(I_{n_1},-I_{n_2},t)) \theta_{\omega_{+}^{\vee}}(diag(I_{n_1},-I_{n_2},t)) dt
\end{eqnarray}
where $J(n)=\{(n_1,n_2,k)\in (\BZ_{\geq 0})^3|\; n_1+n_2+2k=n\}$.

\begin{rmk}
Here we extend the Weyl determinant $D^{\SO_n}(\cdot)$ from $\SO_n(\BR)$ to $\mathrm{O}_n(\BR)$ by the same formula, i.e. for $x\in \mathrm{O}_n(\BR)_{ss}$, we define
$$D^{\SO_n}(x)=|\det(1-Ad(x))|_{\Fs\Fo_n(\BR)/\Fs\Fo_n(\BR)_x}|$$
where $\Fs\Fo_n(\BR)_x$ is the centralizer of $x$ in $\Fs\Fo_n(\BR)$.
\end{rmk}

\begin{rmk}
The reason we consider the model $(\GL_n(\BR),\mathrm{O}_n(\BR))$ is that it behaviors nicely under parabolic induction. To be specific, the intersection of $\mathrm{O}_n(\BR)$ with the standard Levi subgroup $\GL_{n'}(\BR)\times \GL_{n''}(\BR)$ ($n=n'+n''$) of $\GL_n(\BR)$ is $\mathrm{O}_{n'}(\BR)\times \mathrm{O}_{n''}(\BR)$, while intersection of $\SO_n(\BR)$ with $\GL_{n'}(\BR)\times \GL_{n''}(\BR)$ is $\mathrm{S}(\mathrm{O}_{n'}(\BR)\times \mathrm{O}_{n''}(\BR))$.
\end{rmk}

\begin{prop}\label{O implies SO}
Let $\omega_+$ be a finite dimensional representation of $\mathrm{O}_n(\BR)$ and $\omega=\omega_+|_{\SO_n(\BR)}$ which is a finite dimensional representation of $\SO_n(\BR)$. For all finite length smooth representations $\pi$ of $\GL_n(\BR)$, we have
$$m(\pi,\omega)=m(\pi,\omega_+)+m(\pi,\omega_+\otimes sgn),\; m_{geom}(\pi,\omega)=m_{geom}(\pi,\omega_+)+m_{geom}(\pi,\omega_+\otimes sgn).$$
\end{prop}

\begin{proof}
The second equation follows from the definitions of $m_{geom}(\pi,\omega)$ and $m_{geom}(\pi,\omega_+)$, together with the fact that $\theta_{\omega_{+}^{\vee}\otimes sgn}(h)=\theta_{\omega_{+}^{\vee}}(h)sgn(h)$ for all $h\in \mathrm{O}_n(\BR)$.

For the first equation, we just need to show that the linear map
$$\Hom_{\mathrm{O}_n(\BR)}(\pi,\omega_+)\oplus \Hom_{\mathrm{O}_n(\BR)}(\pi,\omega_+\otimes sgn)\rightarrow \Hom_{\SO_n(\BR)}(\pi,\omega):\; l_1\oplus l_2\mapsto l_1+l_2$$
is an isomorphism. It is clear that this map is injective, so we just need to show that it is surjective. Given $l\in \Hom_{\SO_n(\BR)}(\pi,\omega)$, we have $l=\frac{l_1+l_2}{2}$ where
$$l_1=l+\omega_+(\varepsilon)^{-1}\circ l\circ \pi(\varepsilon),\;l_2=l-\omega_+(\varepsilon)^{-1}\circ l\circ \pi(\varepsilon),\; \varepsilon=diag(-1,I_{n-1})\in \mathrm{O}_n(\BR)-\SO_n(\BR).$$
It is enough to show that
$$l_1\in \Hom_{\mathrm{O}_n(\BR)}(\pi,\omega_+),\; l_2\in \Hom_{\mathrm{O}_n(\BR)}(\pi,\omega_+\otimes sgn).$$

For $v\in \pi$ and $h\in \SO_n(\BR)$, we have
$$l_1(\pi(h)v)=l(\pi(h)v) + \omega_+(\varepsilon)^{-1}\big(l(\pi(\varepsilon h )v) \big) =\omega(h)l(v)+ \omega_+(\varepsilon)^{-1} \big(l(\pi(\varepsilon h \varepsilon^{-1}) \pi(\varepsilon) v) \big)$$
$$=\omega(h)l(v)+ \omega_+(\varepsilon)^{-1} \big(\omega(\varepsilon h \varepsilon^{-1}) l( \pi(\varepsilon) v) \big)=\omega(h)l(v)+\omega(h) \omega_+(\varepsilon)^{-1} l(\pi(\varepsilon)v)=\omega(h)l_1(v)$$
and
$$l_1(\pi(\varepsilon)v)=l(\pi(\varepsilon)v)+\omega_+(\varepsilon)^{-1}\big(l(\pi(\varepsilon^2)v) \big)=l(\pi(\varepsilon)v)+ \omega_+(\varepsilon)^{-1}\big (\omega(\varepsilon^2) l(v) \big)$$
$$=l(\pi(\varepsilon)v)+\omega_+(\varepsilon) l(v) =\omega_+(\varepsilon) l_1(v).$$
This implies that $l_1\in \Hom_{\mathrm{O}_n(\BR)}(\pi,\omega_+)$. Similarly, we can also show that $l_2\in \Hom_{\mathrm{O}_n(\BR)}(\pi, \omega_+\otimes sgn)$. This proves the proposition.
\end{proof}

The following theorem will be proved in the next subsection. It gives a multiplicity formula for the model $(\GL_n(\BR),\mathrm{O}_n(\BR))$.

\begin{thm}\label{main O(n)}
For all finite length smooth representations $\pi$ of $\GL_n(\BR)$ and for all finite dimensional representations $\omega_+$ of $\mathrm{O}_n(\BR)$, we have
\begin{equation}\label{3.2}
m(\pi,\omega_+)=m_{geom}(\pi,\omega_+).
\end{equation}
\end{thm}

Now we are ready to prove \eqref{3.1}. It is enough to consider the case when $\omega$ is irreducible. We use $\omega'$ to denote the irreducible representation of $\SO_n(\BR)$ given by $\omega'(h)=\omega(\varepsilon^{-1} h \varepsilon)$ with $\varepsilon=diag(-1,I_{n-1})$. If $\omega\simeq \omega'$, there exists an irreducible representation $\omega_+$ of $\mathrm{O}_n(\BR)$ such that $\omega=\omega_+|_{\SO_n(\BR)}$. Then \eqref{3.1} follows from Proposition \ref{O implies SO} and Theorem \ref{main O(n)}.

If $\omega$ is not isomorphic to $\omega'$ (this only happens when $n$ is even), then there exists an irreducible representation $\omega_+$ of $\mathrm{O}_n(\BR)$ such that $\omega\oplus \omega'=\omega_+|_{\SO_n(\BR)}$. By Proposition \ref{O implies SO} and Theorem \ref{main O(n)}, we have
$$m(\pi,\omega)+m(\pi,\omega')=m_{geom}(\pi,\omega)+m_{geom}(\pi,\omega').$$
Hence in order to prove \eqref{3.1}, it is enough to show that
$$m(\pi,\omega)=m(\pi,\omega'),\; m_{geom}(\pi,\omega)=m_{geom}(\pi,\omega').$$
The first equation follows from the fact that the linear map
$$\Hom_{\SO_n(\BR)}(\pi,\omega)\rightarrow \Hom_{\SO_n(\BR)}(\pi,\omega'):\; l\mapsto \omega_+(\varepsilon)^{-1}\circ l$$
is an isomorphism. The second equation follows from the facts that $\theta_{\omega^{\vee}}(h)=\theta_{(\omega')^{\vee}}(\varepsilon^{-1} h \varepsilon)$ for all $h\in \SO_n(\BR)$ and $\theta_{\pi}$ is invariant under $\varepsilon$-conjugation. This finishes the proof of \eqref{3.1} and hence the proof of Theorem \ref{main K-type}(1).

\subsection{The proof of Theorem \ref{main O(n)}}
In this subsection, we are going to prove Theorem \ref{main O(n)}. To simplify the notation, we will replace $\omega_+$ by $\omega$. We first consider the cases when $n=2$ (the case when $n=1$ is trivial). We need to show that for all smooth finite length representations $\pi$ of $\GL_2(\BR)$ and for all finite dimensional representations $\omega$ of $\mathrm{O}_2(\BR)$, we have
\begin{eqnarray}\label{GL(2)}
m(\pi,\omega)=m_{geom}(\pi,\omega)&:=&\frac{c_{\pi}(I_2)\theta_{\omega}(I_2)+c_{\pi}(-I_2)\theta_{\omega^{\vee}}(-I_2)+2\theta_{\pi}
(\begin{pmatrix}1&0\\0&-1 \end{pmatrix})  \theta_{\omega^{\vee}}(\begin{pmatrix} 1 &0\\ 0&-1 \end{pmatrix})}{4}  \nonumber \\
&& + \frac{1}{2}\int_{\SO_2(\BR)} \theta_{\pi}(t) \theta_{\omega^{\vee}}(t) dt.
\end{eqnarray}

When $\pi$ is finite dimensional, by the representation theory of compact groups, we have
$$m(\pi,\omega)=\int_{\mathrm{O}_2(\BR)} \theta_{\pi}(t)\theta_{\omega^{\vee}}(t)dt=\frac{\theta_{\pi}(\begin{pmatrix}1&0\\0&-1 \end{pmatrix})  \theta_{\omega^{\vee}}(\begin{pmatrix} 1 &0\\ 0&-1 \end{pmatrix})}{2}+\frac{1}{2}\int_{\SO_2(\BR)} \theta_{\pi}(t) \theta_{\omega^{\vee}}(t) dt.$$
Here the Haar measure on $\mathrm{O}_2(\BR)$ (resp. $\SO_n(\BR)$) is choosen so that the total volume is equal to 1. On the other hand, since $\pi$ is finite dimensional, we have $c_{\pi}(I_2)=c_{\pi}(-I_2)=0$. This proves \eqref{GL(2)}.

Then we consider the induced representations. Assume that $\pi=I_{B}^{\GL_2}(\pi_1\otimes \pi_2)$ where $B=TN$ is the upper triangular Borel subgroup of $\GL_2(\BR)$ and $\pi_1\otimes \pi_2$ is a finite dimensional representation of $T(\BR)=\GL_1(\BR)\times \GL_1(\BR)$. By the Iwasawa decomposition $\GL_2(\BR)=B(\BR)\mathrm{O}_2(\BR)$ and the reciprocity law, we have
$$\Hom_{\mathrm{O}_2(\BR)}(\pi,\omega)=\Hom_{\mathrm{O}_1(\BR)\times \mathrm{O}_1(\BR)}(\pi_1\otimes \pi_2,\omega|_{\mathrm{O}_1(\BR)\times \mathrm{O}_1(\BR)}).$$
By the representation theory of finite group (note that $\mathrm{O}_1(\BR)=\BZ/2\BZ$ is a finite group), we have
\begin{eqnarray*}
m(\pi,\omega)&=&\frac{\theta_{\pi_1}(1)\theta_{\pi_2}(1)\theta_{\omega^{\vee}}(I_2)}{4} + \frac{\theta_{\pi_1}(-1)\theta_{\pi_2}(-1)\theta_{\omega^{\vee}}(-I_2)}{4}+\\
&& \frac{\theta_{\pi_1}(1) \theta_{\pi_2}(-1) \theta_{\omega^{\vee}}(\begin{pmatrix}1&0\\0&-1 \end{pmatrix})}{4}+ \frac{\theta_{\pi_1}(-1)\theta_{\pi_2}(1) \theta_{\omega^{\vee}}(\begin{pmatrix}-1&0\\0&1 \end{pmatrix})}{4}.
\end{eqnarray*}
On the other hand, by Proposition \ref{germ parabolic induction}, we have
\begin{eqnarray*}
m_{geom}(\pi,\omega)&=&\frac{\theta_{\pi_1}(1)\theta_{\pi_2}(1)\theta_{\omega^{\vee}}(I_2)}{4} + \frac{\theta_{\pi_1}(-1)\theta_{\pi_2}(-1)\theta_{\omega^{\vee}}(-I_2)}{4}+\\
&& \frac{\theta_{\pi_1}(1) \theta_{\pi_2}(-1) \theta_{\omega^{\vee}}(\begin{pmatrix}1&0\\0&-1 \end{pmatrix})}{4}+ \frac{\theta_{\pi_1}(-1)\theta_{\pi_2}(1) \theta_{\omega^{\vee}}(\begin{pmatrix}-1&0\\0&1 \end{pmatrix})}{4}.
\end{eqnarray*}
This proves \eqref{GL(2)}.

Now we can prove \eqref{GL(2)} for the general case. It is enough to consider the case when $\pi$ is irreducible. There are three kinds of irreducible smooth representation of $\GL_2(\BR)$: finite dimensional representation, principal series and discrete series. The first two cases have already been considered, so it remains to consider the discrete series case. Assume that $\pi$ is an irreducible discrete series. Then there exists a character $\chi_1\otimes \chi_2$ of $T(\BR)=\GL_1(\BR)\times \GL_1(\BR)$ such that $\pi$ is the unique subrepresentation of $\Pi=I_{B}^{\GL_2}(\chi_1\otimes \chi_2)$ and $\pi'=\Pi/\pi$ is a finite dimensional representation of $\GL_2(\BR)$. We have
$$m(\Pi,\omega)=m(\pi,\omega)+m(\pi',\omega),\; m_{geom}(\Pi,\omega)=m_{geom}(\pi,\omega)+m_{geom}(\pi',\omega).$$
By the discussion above, we have $m(\Pi,\omega)=m_{geom}(\Pi,\omega)$ and $m(\pi',\omega)=m_{geom}(\pi',\omega)$. Hence \eqref{GL(2)} also holds for discrete series. This proves Theorem \ref{main O(n)} when $n=2$.

Now assume that $n>2$, we are going to prove Theorem \ref{main O(n)} for $\GL_n(\BR)$. By induction, we assume that Theorem \ref{main O(n)} holds for $\GL_k(\BR)$ when $k<n$. By Proposition \ref{GL(n) induction}, in order to prove Theorem \ref{main O(n)}, it is enough to prove the following proposition.

\begin{prop}
Theorem \ref{main O(n)} holds for all induced representations. In other words, if $\pi=I_{P}^{\GL_n}(\tau)$ is an induced representation with $P=MN$ be a proper parabolic subgroup of $\GL_n$ and $\tau$ be a finite length smooth representation of $M(\BR)$, then $m(\pi,\omega)=m_{geom}(\pi,\omega)$ for all smooth finite dimensional representations $\omega$ of $\mathrm{O}_n(\BR)$.
\end{prop}

\begin{proof}
Let $\pi$ be an induced representation of $\GL_n(\BR)$. Then there exists a maximal upper triangular parabolic subgroup $P=MN$ of $\GL_n(\BR)$ and a finite length smooth representation $\tau$ of $M(\BR)$ such that $\pi=I_{P}^{\GL_n}(\tau)$. Since $P$ is maximal, $M(\BR)=\GL_{n'}(\BR)\times \GL_{n''}(\BR)$ for some $n',n''>0$ with $n=n'+n''$ and $\tau=\tau'\otimes \tau''$ where $\tau'$ (resp. $\tau''$) is a finite length smooth representation of $\GL_{n'}(\BR)$ (resp. $\GL_{n''}(\BR)$).

By the Iwasawa decomposition $\GL_n(\BR)=P(\BR)\mathrm{O}_n(\BR)$ and the reciprocity law, we have
$$\Hom_{\mathrm{O}_n(\BR)}(\pi,\omega)\simeq \Hom_{\mathrm{O}_{n'}(\BR)\times \mathrm{O}_{n''}(\BR)}(\tau_1\otimes \tau_2, \omega|_{\mathrm{O}_{n'}(\BR)\times \mathrm{O}_{n''}(\BR)}).$$
Together with the inductional hypothesis (applied to the pairs $(\GL_{n'}(\BR),\mathrm{O}_{n'}(\BR))$ and $(\GL_{n''}(\BR),\mathrm{O}_{n''}(\BR))$), we have
\begin{eqnarray*}
m(\pi,\omega)&=&\sum_{(n_1',n_2',k')\in J(n'),(n_1'',n_2'',k'')\in J(n'')} \frac{1}{2^{n'-k'} k'!} \frac{1}{2^{n''-k''} k''!} \int_{(\BC^1)^{k'}} \int_{(\BC^1)^{k''}}   \\
&& D^{\SO_{n'}}(diag(I_{n_1'},-I_{n_2'},t')) D^{\SO_{n''}}(diag(I_{n_1''},-I_{n_2''},t''))  c_{\pi'}(diag(I_{n_1'},-I_{n_2'},t'))
\end{eqnarray*}
\begin{equation}\label{3.4}
\;\;\;\;\;c_{\pi''}(diag(I_{n_1''},-I_{n_2''},t'')) \theta_{\omega^{\vee}}(diag(I_{n_1'},-I_{n_2'},t',I_{n_1''},-I_{n_2''},t'')) dt'dt''.
\end{equation}
It remains to show that $m_{geom}(\pi,\omega)$ is equal to the right hand side of \eqref{3.4}.

We first recall the definition of $m_{geom}(\pi,\omega)$ from \eqref{3.3}:
\begin{eqnarray}\label{3.5}
m_{geom}(\pi,\omega)&=&\sum_{(n_1,n_2,k)\in J(n)} \frac{1}{2^{n-k} k!} \int_{(\BC^1)^k} D^{\SO_n}(diag(I_{n_1},-I_{n_2},t)) \nonumber \\ &&c_{\pi}(diag(I_{n_1},-I_{n_2},t)) \theta_{\omega^{\vee}}(diag(I_{n_1},-I_{n_2},t)) dt.
\end{eqnarray}
For $(n_1,n_2,k)\in J(n)=\{(n_1,n_2,k)\in (\BZ_{\geq 0})^3|\; n_1+n_2+2k=n\}$, let
\begin{eqnarray*}
I(n_1,n_2,k)&=&\{(n_1',n_1'',n_2',n_2'',k',k'')\in \BZ_{\geq 0}^{6}|\; n_1=n_1'+n_1'',n_2=n_2'+n_2'',k=k'+k'',\\
&&(n_1',n_2',k')\in J(n'),(n_1'',n_2'',k'')\in J(n'')\}.
\end{eqnarray*}
By Proposition \ref{germ parabolic induction}, for $(n_1,n_2,k)\in J(n)$ and $t=t_1\times t_2\times \cdots \times t_k\in (\BC^1)^k$ with $t_i\neq \pm1$, $t_i\neq t_j$ and $t_i\neq \overline{t_j}$ for $1\leq i\neq j\leq n$, we have
\begin{equation}\label{3.6}
D^{\SO_n}(diag(I_{n_1},-I_{n_2},t))c_{\pi}(diag(I_{n_1},-I_{n_2},t))=\sum_{(n_1',n_1'',n_2',n_2'',k',k'')\in I(n_1,n_2,k)}\sum_{\{i_1,\cdots,i_{k'}\},\{j_1,\cdots,j_{k''}\}}
\end{equation}
$$D^{\SO_{n'}}(diag(I_{n_1'},-I_{n_2'},t')) D^{\SO_{n''}}(diag(I_{n_1''},-I_{n_2''},t''))c_{\pi'}(diag(I_{n_1'},-I_{n_2'},t')) c_{\pi''}(diag(I_{n_1''},-I_{n_2''},t''))$$
where
\begin{itemize}
\item $i_1<i_2<\cdots <i_{k'},\; j_1<j_2<\cdots<j_{k''}$. $\{i_1,\cdots,i_{k'}\}$ runs over the subsets of $\{1,2,\cdots,k\}$ containing $k'$-many elements and $\{j_1,\cdots,j_{k''}\}=\{1,2,\cdots,k\}-\{i_1,\cdots,i_{k'}\}$.
\item $t'=t_{i_1}\times t_{i_2}\times \cdots \times t_{i_{k'}}$ and $t''=t_{j_1}\times t_{j_2}\times \cdots \times t_{j_{k''}}$.
\end{itemize}

Combining \eqref{3.4}, \eqref{3.5} and \eqref{3.6}, we have $m(\pi,\omega)=m_{geom}(\pi,\omega)$. This finishes the proof of the proposition and hence the proofs of Theorem \ref{main K-type}(1) and \ref{main O(n)}.
\end{proof}

\section{The proof of Theorem \ref{main K-type}(2)}
In this section, let $H$ be a connected reductive group defined over $\BR$ with $H(\BR)$ compact and let $G=Res_{\BC/\BR}H$. Let $\pi$ be a finite length smooth representation of $G(\BR)$ and $\omega$ be a finite dimensional representation of $H(\BR)$. We have defined the multiplicity
$$m(\pi,\omega)=\dim(\Hom_{H(\BR)}(\pi,\omega))$$
in previous sections. Moreover, by the discussion in Section \ref{section Galois}, we know that the geometric multiplicity in this case is defined by
$$m_{geom}(\pi,\omega)=|W(H,T)|^{-1} \int_{T(\BR)} D^H(t)\theta_{\pi}(t)\theta_{\omega^{\vee}}(t)dt=|W(G)|^{-1} \int_{T(\BR)} D^H(t)\theta_{\pi}(t)\theta_{\omega^{\vee}}(t)dt$$
where $T(\BR)$ is a maximal torus of $H(\BR)$ (which is unique up to $H(\BR)$-conjugation) and $W(H,T)$ is the Weyl group which is isomorphic to the Weyl group $W(G)$ of $G(\BR)=H(\BC)$. The goal of this section is to prove Theorem \ref{main K-type}(2). In other words, we need to show that
\begin{equation}\label{4.1}
m(\pi,\omega)=m_{geom}(\pi,\omega).
\end{equation}

When $G$ is abelian, \eqref{4.1} is trivial. Hence by induction, we may assume that \eqref{4.1} holds for all the proper Levi subgroups of $G$. By Proposition \ref{GL(n) induction}, it is enough to prove the following proposition.

\begin{prop}
\eqref{4.1} holds for all induced representations. In other words, if $\pi=I_{P}^{G}(\tau)$ is an induced representation with $P=MN$ be a proper parabolic subgroup of $G$ and $\tau$ be a finite length smooth representation of $M(\BR)$, then $m(\pi,\omega)=m_{geom}(\pi,\omega)$ for all finite dimensional representations $\omega$ of $H(\BR)$.
\end{prop}

\begin{proof}
By conjugating $M$ we may assume that $P(\BR)\cap H(\BR)=M(\BR)\cap H(\BR)$ is a maximal compact subgroup of $M(\BR)$. Set $H_M=M\cap H$, then $M\simeq Res_{\BC/\BR}H_M$. Moreover, we may choose the torus $T$ so that $T\subset H_M$ (i.e. $T(\BR)$ is also a maximal torus of $H_M(\BR)$). By the Iwasawa decomposition $G(\BR)=P(\BR)H(\BR)$ and the reciprocity law, we have
$$\Hom_{H(\BR)}(\pi,\omega)\simeq \Hom_{H_M(\BR)}(\tau,\omega|_{H_M(\BR)}).$$
Combining with our inductional hypothesis (applied to the pair $(M(\BR),H_M(\BR))$), we have
\begin{equation}\label{4.2}
m(\pi,\omega)=|W(M)|^{-1} \int_{T(\BR)}D^{H_M}(t)\theta_{\tau}(t) \theta_{\omega^{\vee}}(t) dt
\end{equation}
where $W(M)$ is the Weyl group of $M(\BR)=H_M(\BC)$.

For $t\in T(\BR)\cap G_{reg}(\BR)$, we have $D^H(t)=D^G(t)^{1/2}$ and $D^{H_M}(t)=D^M(t)^{1/2}$. Combining with Proposition \ref{germ parabolic induction}, we have
$$D^{H}(t)\theta_{\pi}(t)=\sum_{t_M}  D^{H_M}(t_M) \theta_{\tau}(t_M)$$
where $t_M$ runs over a set of representatives for the $M(\BR)$-conjugacy classes of elements in $T(\BR)$ that are $G(\BR)$-conjugated to $t$. As a result, we have
\begin{equation}\label{4.3}
\int_{T(\BR)} D^H(t)\theta_{\pi}(t)\theta_{\omega^{\vee}}(t)dt= \frac{|W(G)|}{|W(M)|} \int_{T(\BR)} D^{H_M}(t)\theta_{\tau}(t)\theta_{\omega^{\vee}}(t)dt.
\end{equation}
Now the proposition follows from \eqref{4.2} and \eqref{4.3}. This finishes the proof of the proposition and the proof of Theorem \ref{main K-type}(2).
\end{proof}

\end{document}